\documentclass[a4paper,11pt]{article}
\pdfoutput=1
\usepackage{hyperref}
\hypersetup{hypertexnames = false, bookmarksdepth = 2, bookmarksopen = true, colorlinks, linkcolor = black, citecolor = black, urlcolor = black, pdfstartview={XYZ null null 1}}

\usepackage[table,dvipsnames]{xcolor}

\usepackage{amsfonts}
\usepackage[fleqn, leqno]{amsmath}
\usepackage{amsthm}

\usepackage[capitalise]{cleveref}

\usepackage[utf8]{inputenc}
\usepackage[backend=biber, maxbibnames=99, datamodel=mrnumber, sortcites]{biblatex}

\usepackage{booktabs}
\usepackage{colonequals}
\usepackage{diagbox}
\usepackage{enumitem}
\usepackage{mathtools}
\usepackage{parskip}
\usepackage{thmtools}
\usepackage{tikz-cd}
\usepackage[colorinlistoftodos, textsize = footnotesize]{todonotes}
\usepackage{xparse}
\usepackage{xspace}
\usepackage{youngtab}

\usepackage[T1]{fontenc}
\usepackage{libertine}
\usepackage[libertine]{newtxmath}
\usepackage[scaled=0.83]{beramono}
\usepackage{eucal}
\usepackage{microtype}
\usepackage{gitinfo2}

\usepackage[frozencache,cachedir=minted-cache]{minted}

\DeclareFieldFormat{mrnumber}{%
  MR\addcolon\space
  \ifhyperref
    {\href{http://www.ams.org/mathscinet-getitem?mr=#1}{\nolinkurl{#1}}}
    {\nolinkurl{#1}}}

\renewbibmacro*{doi+eprint+url}{%
  \iftoggle{bbx:doi}
    {\printfield{doi}}
    {}%
  \newunit\newblock
  \printfield{mrnumber}%
  \newunit\newblock
  \iftoggle{bbx:eprint}
    {\usebibmacro{eprint}}
    {}%
  \newunit\newblock
  \iftoggle{bbx:url}
    {\usebibmacro{url+urldate}}
    {}}

\newcounter{todocounter}
\DeclareDocumentCommand\addreference{g}{\stepcounter{todocounter}\todo[color = blue!30]{\thetodocounter. Add reference\IfNoValueF{#1}{: #1}}\xspace}
\DeclareDocumentCommand\checkthis{g}{\stepcounter{todocounter}\todo[color = red!50]{\thetodocounter. Check this\IfNoValueF{#1}{: #1}}\xspace}
\DeclareDocumentCommand\fixthis{g}{\stepcounter{todocounter}\todo[color = orange!50]{\thetodocounter. Fix this\IfNoValueF{#1}{: #1}}\xspace}
\DeclareDocumentCommand\expand{g}{\stepcounter{todocounter}\todo[color = green!50]{\thetodocounter. Expand\IfNoValueF{#1}{: #1}}\xspace}

\declaretheoremstyle[
  spaceabove = 3pt,
  spacebelow = 3pt,
  bodyfont = \itshape,
]{first}
\declaretheoremstyle[
  spaceabove = 3pt,
  spacebelow = 3pt,
]{second}
\declaretheorem[numberwithin=section, style=first]{theorem}

\declaretheorem[sibling=theorem, style=first]{lemma}
\declaretheorem[sibling=theorem, style=first]{proposition}

\declaretheorem[sibling=theorem, style=second]{example}
\declaretheorem[sibling=theorem, style=second]{remark}
\declaretheorem[sibling=theorem, style=second]{definition}

\declaretheorem[sibling=theorem, style=second]{construction}

\crefname{construction}{Construction}{Constructions}

\declaretheorem[numberwithin=section, style=first, title=Theorem]{alphatheorem}

\declaretheorem[sibling=alphatheorem, style=first, title=Conjecture]{alphaconjecture}

\declaretheorem[sibling=alphatheorem, style=first, title=Proposition]{alphaproposition}

\crefname{alphatheorem}{Theorem}{Theorems}
\crefname{alphaconjecture}{Conjecture}{Conjectures}
\crefname{alphacorollary}{Corollary}{Corollaries}
\crefname{alphaproposition}{Proposition}{Propositions}
\crefname{alphaassumption}{Assumption}{Assumptions}

\crefname{figure}{Figure}{Figures}

\makeatletter
\def\gitfootnote{\gdef\@thefnmark{}\@footnotetext}
\makeatother

\newtoggle{bbx:boldentries}
\DeclareBibliographyCategory{boldentry}
\AtEveryBibitem{\ifboolexpr{togl {bbx:boldentries} and test {\ifcategory{boldentry}}}{\bfseries}{}}

\toggletrue{bbx:boldentries}

\mathchardef\mhyphen="2D
\newcommand\dash{\nobreakdash-\hspace{0pt}}

\DeclareMathOperator\coh{coh}
\DeclareMathOperator\derived{\mathbf{D}}
\DeclareMathOperator\Ext{Ext}
\DeclareMathOperator\fano{F}
\DeclareMathOperator\FL{FL}
\DeclareMathOperator\Gr{Gr}
\DeclareMathOperator\HH{H}
\DeclareMathOperator\HHHH{HH}

\DeclareMathOperator\IC{IC}
\DeclareMathOperator\Jac{Jac}
\DeclareMathOperator\Kzero{K_0}
\DeclareMathOperator\OGr{OGr}

\DeclareMathOperator\rk{rk}
\DeclareMathOperator\Sym{Sym}

\newcommand\bounded{\ensuremath{\mathrm{b}}}
\newcommand\Cat{\ensuremath{\mathrm{Cat}}}
\newcommand\MHM{\ensuremath{\mathrm{MHM}}}
\newcommand\clifford{\ensuremath{\mathcal{C}\ell}}

\newcommand\lefschetz{\ensuremath{\mathbb{L}}}
\newcommand\rs{\ensuremath{\mathrm{reg\,ss}}}

\newcommand\Var{\ensuremath{\mathrm{Var}}}

\newcommand\field{\mathbf{k}}
\newcommand\hodge{\ensuremath{\mathrm{Hdg}}}
\newcommand\hochschild{\ensuremath{\mathrm{Hoch}}}

\newcommand\rH{\ensuremath{\mathrm{H}}}
\newcommand\FS{\ensuremath{\mathrm{FS}}}

\newcommand\gauss[2]{\ensuremath{\binom{#1}{#2}_{\lefschetz}}}

\newcommand\fonarev[1]{\ensuremath{\smash{\widetilde{\Sym}}^{#1}\mathcal{C}}}

\DeclareFontFamily{U}{mathx}{}
\DeclareFontShape{U}{mathx}{m}{n}{<-> mathx10}{}
\DeclareSymbolFont{mathx}{U}{mathx}{m}{n}
\DeclareMathAccent{\widehat}{0}{mathx}{"70}
\DeclareMathAccent{\widecheck}{0}{mathx}{"71}

\usepackage{fullpage}

\title{On decompositions for Fano schemes \\ of intersections of two quadrics}
\author{
  Pieter Belmans \and
  Jishnu Bose \and
  Sarah Frei \and
  Benjamin Gould \and
  James Hotchkiss \and
  Alicia Lamarche \and
  Jack Petok \and
  Cristian Rodriguez Avila \and
  Saket Shah
}

\begin{document}
\maketitle

\begin{abstract}
  We propose conjectural semiorthogonal decompositions for
  Fano schemes of linear subspaces on intersections of two quadrics,
  in terms of symmetric powers of the associated hyperelliptic (resp.~stacky) curve.
  When the intersection is odd-dimensional,
  we moreover conjecture an identity in the Grothendieck ring of varieties
  and other motivic contexts.
  The evidence for these conjectures is given by
  upgrading recent results of Chen--Vilonen--Xue,
  to obtain formulae for the Hodge numbers of these Fano schemes.
  This allows us to numerically verify the conjecture
  in the hyperelliptic case,
  and establish a combinatorial identity as evidence for the stacky case.
\end{abstract}

\section{Introduction}
Let~$Q_1\cap Q_2$ be the intersection of two quadrics in~$\mathbb{P}^{2g+1}$,
resp.~$\mathbb{P}^{2g}$,
for~$g\geq 2$,
where we work over a field~$\field$
which is algebraically closed of characteristic~zero.
When topological methods are used,
we take~$\field=\mathbb{C}$.
It is well-known that there is an associated hyperelliptic (resp.~stacky) curve
which controls much of the geometry of~$Q_1\cap Q_2$,
as recalled in \cref{subsection:curves}.
We denote this curve by~$C$ (resp.~$\mathcal{C}$).

Our starting point
is the semiorthogonal decomposition for
the derived category of~$Q_1\cap Q_2$
due to Bondal--Orlov \cite[Theorem~2.9]{alg-geom/9506012} (for the hyperelliptic case)
resp.~Kuznetsov \cite[Corollary~5.7]{MR2419925} (for both cases)
which reads
\begin{equation}
  \label{equation:bondal-orlov-kuznetsov-sod}
  \derived^\bounded(Q_1\cap Q_2)
  =
  \begin{cases}
    \langle\derived^\bounded(C),\mathcal{O}_{Q_1\cap Q_2},\ldots,\mathcal{O}_{Q_1\cap Q_2}(2g-3)\rangle & \dim Q_1\cap Q_2=2g-1 \\
    \langle\derived^\bounded(\mathcal{C}),\mathcal{O}_{Q_1\cap Q_2},\ldots,\mathcal{O}_{Q_1\cap Q_2}(2g-4)\rangle & \dim Q_1\cap Q_2=2g-2.
  \end{cases}
\end{equation}

\paragraph{Fano schemes of intersections of two quadrics}
We will generalize from~$Q_1\cap Q_2$ and consider the Fano schemes of~$k$\dash dimensional linear subspaces on~$Q_1\cap Q_2$,
as introduced in \cite{MR0569043},
and previously studied in \cite{reid-thesis}.
They will be denoted~$\fano_k(Q_1\cap Q_2)$,
where~$k=0$ refers to the Fano scheme of points on~$Q_1\cap Q_2$,
i.e., $Q_1\cap Q_2$ itself,
and~$k=g-1$ is the highest dimension for which they are non-empty.
These Fano schemes are smooth projective varieties of expected dimension~$(k+1)(2g+1-k)-2\binom{k+2}{k}$
resp.~$(k+1)(2g-k)-2\binom{k+2}{k}$.
We recall some more of their geometric properties in \cref{subsection:fano-schemes}.

For~$Q_1\cap Q_2\subset\mathbb{P}^{2g+1}$,
the Fano scheme~$\fano_{g-2}(Q_1\cap Q_2)$
is the moduli space~$\mathrm{M}_C(2,\mathcal{L})$
of stable rank-two bundles with fixed determinant of odd degree \cite[Theorem~1]{MR0429897},
the Fano scheme~$\fano_{g-1}(Q_1\cap Q_2)$
is the Jacobian~$\Jac(C)$ \cite[Theorem~4.8]{reid-thesis},
and for all~$k=0,\ldots,g-1$
there is an interpretation as a moduli space of orthogonal bundles \cite[Theorem~3]{MR0653952}.

The goal of this article is to propose conjectural semiorthogonal decompositions for~$\derived^\bounded(\fano_k(Q_1\cap Q_2))$,
generalizing from the case \eqref{equation:bondal-orlov-kuznetsov-sod} for~$k=0$,
and give evidence for these conjectures.

\paragraph{Hyperelliptic case}
We first consider the \emph{odd-dimensional} case,
i.e.,~$Q_1\cap Q_2\subset\mathbb{P}^{2g+1}$,
so that the associated curve~$C$ is a hyperelliptic curve.
In what follows,
$\Sym^iC$ denotes the~$i$th symmetric power of~$C$,
which is a smooth projective variety of dimension~$i$.
\begin{alphaconjecture}
  \label{conjecture:hyperelliptic}
  Let~$Q_1\cap Q_2$ be a smooth intersection of quadrics in~$\mathbb{P}^{2g+1}$,
  and let~$C$ be the associated hyperelliptic curve.
  For all~$k=0,\ldots,g-2$ there exists a semiorthogonal decomposition
  \begin{equation}
    \label{equation:hyperelliptic-conjecture}
    \derived^\bounded(\fano_k(Q_1\cap Q_2))
    =
    \left\langle
    \text{$\binom{2g-4-k-i}{k+1-i} + 2\binom{2g-4-k-i}{k-i}$ copies of $\derived^\bounded(\Sym^iC)$, for $i=0,\ldots,k+1$}
    \right\rangle.
  \end{equation}
\end{alphaconjecture}

The indecomposability of~$\derived^\bounded(\Sym^iC)$ for~$i\leq g-1$ follows from \cite[Theorem~1.9]{2107.09564v3}.
For~$k=g-1$ we have~$\fano_{g-1}(Q_1\cap Q_2)\cong\Jac(C)$ \cite[Theorem~4.8]{reid-thesis},
whose derived category is also indecomposable.

Our method for giving evidence for this conjecture
uses a (conjectural) identity in the Grothendieck ring of varieties,
and it is for this identity that we will give evidence.
This identity also allows one to conjecture analogous identities,
e.g., in the category of Chow motives,
but we will not explicitly spell these out.

To phrase this conjectural identity for~$[\fano_k(Q_1\cap Q_2)]$
in~$\Kzero(\Var/\field)$ as in \cref{conjecture:hyperelliptic-refined}
we will write
\begin{equation}
  \gauss{n}{m}
  \colonequals
  \frac{(1-\lefschetz^n)(1-\lefschetz^{n-1})\cdots(1-\lefschetz^{n-m+1})}{(1-\lefschetz)(1-\lefschetz^2)\cdots(1-\lefschetz^m)}
\end{equation}
for the Gaussian binomial coefficient
in the variable~$\lefschetz=[\mathbb{A}^1]\in\Kzero(\Var/\field)$.
We define the following element
\begin{equation}
  \label{equation:M}
  \begin{aligned}
    M_{g,k,i}
    &\colonequals
    \lefschetz^{i(g-k-1)}
    \Biggl(
      \gauss{2g-k-i}{k+1-i}
      - \left( \lefschetz^{g-k-1} + \lefschetz^{g+2k-3i} \right)\gauss{2g-k-i-4}{k-i} \\
      &\qquad\qquad\qquad
      - \left( \lefschetz^{g-k} + \lefschetz^{g-i} + \lefschetz^{g+k-2i} + \lefschetz^{3g-3k-4} + \lefschetz^{3g-2k-4-i} + \lefschetz^{3g-k-2i-4} \right)\gauss{2g-k-i-4}{k-i-1} \\
      &\qquad\qquad\qquad
      - \left( \lefschetz^{3(g-k-1)} + \lefschetz^{3(g-k-1)+1} + \lefschetz^{3g-2k-i-3} + \lefschetz^{3g-2k-i-2} \right)\gauss{2g-k-i-4}{k-i-2} \\
      &\qquad\qquad\qquad
      - \lefschetz^{4(g-k)-2}\gauss{2g-k-i-4}{k-i-3}
    \Biggr)
  \end{aligned}
\end{equation}
in the subring~$\mathbb{Z}[\lefschetz]\subset\Kzero(\Var/\field)$.
The following conjecture is phrased for the Grothendieck ring of varieties,
but as explained before,
it can verbatim be phrased in other contexts,
such as an isomorphism in the category of Chow motives.
\begin{alphaconjecture}
  \label{conjecture:hyperelliptic-refined}
  For~$Q_1\cap Q_2$ and~$C$ as in \cref{conjecture:hyperelliptic}
  and~$k=0,\ldots,g-2$
  the expression~$M_{g,k,i}$ in \eqref{equation:M} is effective,
  i.e.,~$M_{g,k,i}\in\mathbb{Z}_{\geq 0}[\lefschetz]$,
  and we have the identity
  \begin{equation}
    \label{equation:hyperelliptic-refined}
    [\fano_k(Q_1\cap Q_2)]
    =
    \sum_{i=0}^{k+1}
    M_{g,k,i}
    [\Sym^iC]
  \end{equation}
  in~$\Kzero(\Var/\field)$.
\end{alphaconjecture}

The evidence for \cref{conjecture:hyperelliptic,conjecture:hyperelliptic-refined}
is discussed in \cref{section:hyperelliptic-decomposition}.
We will explain how to compute the E-polynomial of~$\fano_k(Q_1\cap Q_2)$:
our main result in this direction is the following theorem,
generalizing \cite[Theorem~1.1]{MR3689749} to include
the Hodge structures on both sides of the isomorphism in the statement of loc.~cit.
\begin{alphatheorem}
  \label{theorem:hodge-decomposition}
  For~$Q_1\cap Q_2$ and~$C$ as in \cref{conjecture:hyperelliptic}
  and~$k=0,\ldots,g-2$,
  we will write
  \begin{equation}
    d\colonequals\dim\fano_k(Q_1\cap Q_2)=(k+1)(2g-2k-1).
  \end{equation}
  For all~$i=-d,\ldots,d$
  there exists an isomorphism of Hodge structures
  \begin{equation}
    \HH^{d-i}(\fano_k(Q_1\cap Q_2),\mathbb{Q})
    \cong
    \bigoplus_{j=d-k-1}^g\HH^{g-j}(\Jac C,\mathbb{Q})\otimes\mathbb{Q}(-)^{\oplus N(g-k,j;d-i)}
    \label{equation:isomorphism-hodge-modules}
  \end{equation}
  where~$\mathbb{Q}(-)$ denotes the appropriate Tate twist into weight~$d-i$,
  and the multiplicity~$N(a,b;c)$ is given by the coefficient of~$q^c$ in
  \begin{equation}
    q^{-(b-a+1)(2a-1)}(1-q^{4b})\frac{\prod_{\ell=b-a+2}^{a+b-2}(1-q^{2\ell})}{\prod_{\ell=1}^{2a-2} (1-q^{2\ell})}.
  \end{equation}
\end{alphatheorem}
We do not see an immediate way to turn the multiplicities in \eqref{equation:isomorphism-hodge-modules}
into the multiplicities as they arise in \eqref{equation:hyperelliptic-refined},
a non-trivialc recombination and decomposition step
seems necessary to translate between the two decompositions,
which is why we resort to giving computational evidence in \cref{subsection:hyperelliptic-evidence}.

Having obtained this identity,
we can conjecturally lift the identity in~$\mathbb{Z}[x,y]$
along the E-polynomial motivic measure~$\mu_{\mathrm{E}}$ from \eqref{equation:hodge-polynomial-measure}
to an identity in~$\Kzero(\Var/\field)$,
which leads to \cref{conjecture:hyperelliptic-refined}.

The identity \eqref{equation:hyperelliptic-refined} after applying~$\mu_{\mathrm{E}}$
can be verified computationally for arbitrary~$g$ and~$k$ using \cite{hodge-diamond-cutter},
see \cref{subsection:hyperelliptic-evidence}.
The motivic measure~$\mu_{\derived^\bounded}$ from \eqref{equation:derived-category-measure}
then gives an identity in~$\Kzero(\Cat/\field)$ by setting~$\lefschetz=1$ as in \cref{lemma:multiplicity-reduction}
which thus leads to \cref{conjecture:hyperelliptic}.

\paragraph{Stacky case}
We next consider the \emph{even-dimensional} case,
i.e.,~$Q_1\cap Q_2\subset\mathbb{P}^{2g}$,
so that the associated curve~$\mathcal{C}$ is a stacky curve.
In \cite[\S2]{MR4673249}
an ad hoc definition of a ``stacky symmetric power''~$\Sym^i\mathcal{C}$
is given for the specific stacky curve~$\mathcal{C}$ which appears for~$Q_1\cap Q_2$.
We recall its construction and properties in \cref{subsection:stacky-symmetric-power}.
Unlike the hyperelliptic case,
the derived category of~$\Sym^i\mathcal{C}$ admits a full exceptional collection,
and is thus ``maximally decomposable''.

The analogue of \cref{conjecture:hyperelliptic} in the stacky case is the following.
\begin{alphaconjecture}
  \label{conjecture:stacky}
  Let~$Q_1\cap Q_2$ be a smooth intersection of quadrics in~$\mathbb{P}^{2g}$,
  and let~$\mathcal{C}$ be the associated stacky curve.
  For all~$k=0,\ldots,g-2$ there exists a semiorthogonal decomposition
  \begin{equation}
    \label{equation:stacky-conjecture}
    \derived^\bounded(\fano_k(Q_1\cap Q_2))
    =
    \left\langle
    \text{$\binom{2g-3-k-i}{k+1-i}$ copies of $\derived^\bounded(\Sym^i\mathcal{C})$, for $i=0,\ldots,k+1$}
    \right\rangle.
  \end{equation}
\end{alphaconjecture}
For~$k=g-1$ we have that~$\fano_{g-1}(Q_1\cap Q_2)$
is a reduced and finite scheme of cardinality~$2^{2g}$ \cite[Theorem~3.8]{reid-thesis}.

We present the same kind of evidence as for \cref{conjecture:hyperelliptic,conjecture:hyperelliptic-refined},
by giving the analogue of \cref{theorem:hodge-decomposition}.
The following is (almost) \cite[Theorem~1.1]{MR4216588}.
\begin{alphaproposition}
  \label{proposition:hodge-tate}
  For~$Q_1\cap Q_2$ as in \cref{conjecture:stacky},
  and~$k=0,\ldots,g-2$
  we will write
  \begin{equation}
    d\colonequals\dim\fano_k(Q_1\cap Q_2)=(k+1)(2g-2k-2).
  \end{equation}
  The cohomology of~$\fano_k(Q_1\cap Q_2)$ is concentrated in even degrees,
  and for all~$i=0,\ldots,d$
  there exists an isomorphism of Hodge structures
  \begin{equation}
    \HH^{2i}(\fano_k(Q_1\cap Q_2),\mathbb{Q})
    \cong
    \bigoplus_{j=0}^{k+1}\mathbb{Q}^{\oplus M(i+1,j,k)\binom{2g+1}{j}}
  \end{equation}
  where~$M(a,b;c)$ is the coefficient of~$q^{c-b(n-a)}$ in~$g_{a-b,2g-a-b}(q)$
  with
  \begin{equation}
    g_{d,e}(q)\colonequals\frac{\prod_{\ell=e-d+1}^e(1-q^\ell)}{\prod_{\ell=1}^d(1-q^\ell)}
    =\binom{e}{d}_q.
  \end{equation}
  In particular, the Hodge structure on~$\HH^{2i}(\fano_k(Q_1\cap Q_2),\mathbb{Q})$
  is of Hodge--Tate type.
\end{alphaproposition}
The new ingredient in this statement is that it is of Hodge--Tate type,
the rest is already proven in \cite{MR4216588}.
The Hodge--Tateness is claimed on \cite[page~517]{MR3689749},
but no proof seems to be given in \cite{MR4216588,MR3689749}.

The analogue of \cref{conjecture:hyperelliptic-refined}
expresses~$[\fano_k(Q_1\cap Q_2)]$ entirely as a sum of powers of the Lefschetz motive,
and can be read off from \cref{proposition:hodge-tate}.

\paragraph{Known cases and relations to other works}
Some cases of \cref{conjecture:hyperelliptic,conjecture:stacky} are already known,
and one still open case has been phrased before in the literature:

\begin{itemize}
  \item If~$k=0$
    then \cref{conjecture:hyperelliptic,conjecture:stacky} reduce to
    the known semiorthogonal decompositions \eqref{equation:bondal-orlov-kuznetsov-sod},
    as discussed before.

  \item If~$k=0$
    then \cref{conjecture:hyperelliptic-refined} is known to hold
    up to a common factor of~$\lefschetz$,
    as explained in \cref{lemma:k=0-calculation}.
    For~$g=2$ it is a classical result that it holds without the factor of~$\lefschetz$,
    see \cref{remark:g=2-3-folds}.

  \item If~$k=0,\ldots,g-2$ and~$\dim Q_1\cap Q_2=2g-1$ is odd,
    then \cite[Theorem~2.11]{MR3764066} shows that there is
    a natural fully faithful functor~$\derived^\bounded(C)\hookrightarrow\derived^\bounded(\fano_k(Q_1\cap Q_2))$.
    This is an important first step in settling \cref{conjecture:hyperelliptic}.

  \item If~$k=g-2$ and~$\dim Q_1\cap Q_2=2g-1$ is odd,
    then \cref{conjecture:hyperelliptic} reduces
    to what is referred to as the \emph{BGMN conjecture}.
    In this case~$\fano_{g-2}(Q_1\cap Q_2)\cong\mathrm{M}_C(2,\mathcal{L})$
    is the moduli space of stable rank-two bundles on~$C$ with fixed determinant of odd degree \cite[Theorem~1]{MR0429897}.

    It is in fact possible to let~$C$ be any (and not just hyperelliptic) curve in the statement of the BMGN conjecture.
    The conjecture (together with evidence for it)
    was phrased in \cite{MR4557892,1806.11101v2}.
    Partial progress was obtained in \cite{MR3954042,2106.04872v1},
    whereas the combination of \cite{2108.11951v5,2304.01825v2} settles it for all~$C$.

  \item If~$k=g-2$ and~$\dim Q_1\cap Q_2=2g-2$ is even,
    then \cref{conjecture:stacky} was previously stated in \cite[\S2]{MR4673249}.
    In this case~$\fano_{g-2}(Q_1\cap Q_2)\cong\mathrm{M}_{\mathcal{C}}(2,\mathcal{O}_{\mathcal{C}})$
    is the moduli space of stable quasiparabolic rank-two bundles on~$\mathbb{P}^1$
    -- the coarse space of $\mathcal{C}$ --
    with weights~$1/2$ at the marked points \cite[Theorem]{MR3369361}.
\end{itemize}

\paragraph{Acknowledgements}
We would like to thank the AMS for organising the Mathematics Research Community 2023
``Derived Categories, Arithmetic and Geometry''
where this work was started.
This material is thus based upon work supported by the National Science Foundation under Grant Number DMS 1916439.

We thank Chris Kuo, Kari Vilonen and Calvin Yost-Wolff for interesting conversations. We are particularly grateful to Brad Dirks for his help with the Fourier--Laplace transform.

P.B.~was partially supported by the Luxembourg National Research Fund (FNR--17113194). \\
S.F.~and~J.P.~were partially supported by the Hausdorff Research Institute for Mathematics,
funded by the Deutsche Forschungsgemeinschaft (DFG, German Research Foundation) under Germany's Excellence Strategy – EXC-2047/1 – 390685813.\\
A.L.~was supported by an NSF Postdoctoral Fellowship (\#2103271).

\section{Preliminaries}
Let~$n\geq 2$, and let~$V$ be an~$(n+1)$-dimensional vector space.
We let~$Q_1, Q_2 \subset \mathbb{P}(V)$ be quadric hypersurfaces
such that~$Q_1\cap Q_2\subset\mathbb{P}(V)=\mathbb{P}^n$ is a smooth~$(n-2)$-dimensional complete intersection.
Then by \cite[Proposition~2.1]{reid-thesis} the pencil spanned by~$Q_1$ and~$Q_2$
contains precisely~$n+1$ quadrics of corank~1,
and all others are smooth.
We can thus choose generators~$Q_1$ and~$Q_2$ of the pencil
which are smooth,
and which in an appropriate basis can be written as
\begin{equation}
  \begin{aligned}
    Q_1&\colon x_0^2+\ldots+x_n^2=0 \\
    Q_2&\colon a_0x_0^2+\ldots+a_nx_n^2=0 \\
  \end{aligned}
\end{equation}
for some scalars~$a_i\in\field$, with~$a_i\neq a_j$ for~$i\neq j$.

In what follows we will take~$g\geq 2$, and let~$n+1=2g+2$ or~$2g+1$.
Following \cref{subsection:curves} we will refer to the former as the \emph{hyperelliptic} situation,
and to the latter as the \emph{stacky} situation.

\subsection{Fano schemes of linear subspaces on intersections of two quadrics}
\label{subsection:fano-schemes}
In this paper we are concerned with Fano schemes of linear subspaces.
We will briefly recall their relevant properties.
For~$i=1,2$ we write
\begin{equation}
  \mathrm{OGr}_i(k+1,V)\colonequals\{U \in \mathrm{Gr}(k+1,V)\mid \mathbb{P}(U)\subset Q_i\},
\end{equation}
which is the Grassmannian of isotropic linear subspaces of dimension~$k+1$ in~$V$.
Then for~$0\leq k \leq g-1$ we define
\begin{equation}
  \fano_k(Q_1\cap Q_2)\colonequals \mathrm{OGr}_1(k+1,V)\cap \mathrm{OGr}_2(k+1,V)
\end{equation}
parametrizing projective~$k$-planes contained in~$Q_1\cap Q_2$.
This is the \emph{Fano scheme of linear subspaces on $Q_1\cap Q_2$} as introduced in \cite{MR0569043}.
Alternatively, we can define $\fano_k(Q_1 \cap Q_2)$ as the zero locus of the global section
of~$(\Sym^2 \mathcal{U}^\vee)^{\oplus{2}}$
corresponding to~$Q_1$ and~$Q_2$,
where $\mathcal{U}$ is the tautological subbundle on~$\Gr(k+1,V)$.

We collect some useful facts about $\fano_k(Q_1\cap Q_2)$ here.
We will consider the two possibilities for~$\dim V$ separately,
as the geometry of the Fano schemes is very different.

\paragraph{The hyperelliptic case}
First, suppose $\dim V = 2g+2$, so that $\dim Q_1\cap Q_2 = 2g-1$.
The following combines \cite[Theorem~2.6]{reid-thesis}
and \cite[Remarque~3.2.1, Corollaire~3.5]{MR1654757}.

\begin{proposition}
  \label{proposition:fano-scheme-even-dimension}
  For~$k=0,\ldots,g-1$
  the Fano scheme~$\fano_k(Q_1\cap Q_2)$ is a smooth projective variety of dimension~$(k+1)(2g-2k-1)$.
  Its Picard rank is~1,
  except when~$k=g-1$.
  Its canonical bundle is~$-(2g-2k-2)H$, where~$H$ is the ample generator of the Picard group if~$k\leq g-2$,
  and it is trivial when~$k=g-1$.
\end{proposition}

When~$k=0,\ldots,g-2$ it is thus a Fano variety of index~$2g-2k-2$, and when $k=g-1$~the Fano scheme is Calabi--Yau.
Indeed, in this latter case there is an isomorphism~$\fano_{g-1}(Q_1\cap Q_2)\cong\Jac(C)$ \cite[Theorem~4.8]{reid-thesis},
where $C$ is the hyperelliptic curve which we will introduce in \cref{subsection:curves}.
For~$k=g-2$ the Fano scheme is isomorphic to the moduli space~$\mathrm{M}_C(2,\mathcal{L})$
of stable rank-two bundles with fixed determinant of odd degree \cite[Theorem~1]{MR0429897}.
Finally, for all~$k=0,\ldots,g-1$
the Fano scheme admits an interpretation as a moduli space of orthogonal bundles \cite[Theorem~3]{MR0653952}.

\paragraph{The stacky case}
Next, suppose $\dim V = 2g+1$, so that $\dim Q_1\cap Q_2 = 2g-2$.
The following combines \cite[Theorem~2.6]{reid-thesis}
and \cite[Remarque~3.2.1, Corollaire~3.5]{MR1654757}.

\begin{proposition}
  For~$k=0,\ldots,g-2$,
  the Fano scheme~$\fano_k(Q_1\cap Q_2)$ is a smooth projective variety of dimension~$(k+1)(2g-2k-2)$.
  Its Picard rank is~1,
  except when~$k=g-2$.
  Its canonical bundle is~$-(2g-2k-3)H$,
  where~$H$ is the ample generator of the Picard group if~$k\leq g-3$,
  resp.~an ample class if~$k=g-2$.

  For~$k=g-1$,
  the Fano scheme~$\fano_{g-1}(Q_1\cap Q_2)$ is a finite reduced scheme of length~$4^g$.
\end{proposition}

When~$k=0,\ldots,g-2$ it is thus a Fano variety of index~$2g-2k-3$.

\subsection{Curves associated to intersections of two quadrics}
\label{subsection:curves}
We continue with~$V$ a vector space of dimension~$2g+2$, resp.~$2g+1$,
and take $Q_1, Q_2 \subset \mathbb{P}(V)$ as before:
two smooth quadrics spanning a pencil of quadrics
whose base locus is again smooth.
This pencil will moreover degenerate to quadric cones over~$2g+2$ or $2g+1$ points in the pencil.
To~$Q_1\cap Q_2$ we want to associate a curve which controls the geometry of~$Q_1\cap Q_2$.
Depending on the parity of~$\dim V$ we distinguish two cases.

\paragraph{Hyperelliptic case}
If $\dim V=2g+2$, we consider the hyperelliptic curve $\pi\colon C \to \mathbb{P}^1$
branched over the~$2g+2$ distinguished points in~$\mathbb{P}^1$.
This is the hyperelliptic curve associated to $Q_1\cap Q_2$,
and throughout, we will refer to the case~$\dim V=2g+2$ as the ``hyperelliptic case''.

More intrinsically, this hyperelliptic curve can be described as
the moduli space of families of maximal isotropic subspaces
contained in the fibers of the pencil generated by $Q_1$ and $Q_2$.
Explicitly, if $\mathcal{Q}\subset \mathbb P^1\times \mathbb P^{2g+1}$ is the universal family of quadrics in the pencil,
then $C$ arises in the Stein factorization
\begin{equation}
  \mathcal{F}_{g}(\mathcal{Q}/\mathbb P^1) \to C \to \mathbb P^1,
\end{equation}
where~$\mathcal{F}_{g}(\mathcal{Q}/\mathbb P^1)$ is the relative Fano scheme of projective $g$-planes
in the fibers of $\mathcal{Q} \to \mathbb P^1$.

\paragraph{Stacky case}
If~$\dim V=2g+1$, we consider the root stack~$\mathcal{C}\to\mathbb{P}^1$,
defined by~$\mathbb{Z}/2\mathbb{Z}$-stabilizers at the~$2g+1$ distinguished points in~$\mathbb{P}^1$.
This is the stacky curve associated to~$Q_1\cap Q_2$,
and  throughout, we will refer to the case~$\dim V=2g+1$ as the ``stacky case''.

The original way in which the stacky curve arises from~$Q_1\cap Q_2$ is less direct
than the moduli-theoretic incarnation explained above.
It is shown in \cite[Theorem~5.4]{MR2419925} that
the \emph{homological projective dual} of~$\mathbb{P}^n$
with respect to the Veronese embedding given by~$\mathcal{O}_{\mathbb{P}(V)}(2)$
is the noncommutative variety~$(\mathbb{P}\mathrm{H}^0(\mathbb{P}(V),\mathcal{O}_{\mathbb{P}(V)}(2)),\clifford_0)$
given by the sheaf of even parts of Clifford algebras.
Thus by the formalism of homological projective duality \cite{MR2354207}
the semiorthogonal decomposition of~$Q_1\cap Q_2$
has as its interesting component
the derived category of the restriction of~$\mathcal{C}$ to the projective line
corresponding to the pencil spanned by~$Q_1$ and~$Q_2$.

By \cite[Corollary~3.16]{MR2419925} there exists an equivalence of abelian categories
\begin{equation}
  \coh\mathcal{C}\cong\coh(\mathbb{P}^1,\clifford_0|_{\mathbb{P}^1})
\end{equation}
through the central reduction of \S3.6~in~op.~cit.,
thus~$\mathcal{C}$ only arises a posteriori as a geometric object attached to~$Q_1\cap Q_2$.

In the hyperelliptic case,
a different central reduction is used to go from~$(\mathbb{P}^1,\clifford_0|_{\mathbb{P}^1})$ to~$C$.
Thus, in the proof of \eqref{equation:bondal-orlov-kuznetsov-sod} following \cite{MR2419925},
the hyperelliptic curve~$C$ appears in a more indirect way
than in the original~\cite{alg-geom/9506012}.

\subsection{Grothendieck rings and motivic measures}
In order to give evidence for \cref{conjecture:hyperelliptic,conjecture:hyperelliptic-refined,conjecture:stacky}
we will consider the Grothendieck ring of varieties (resp.~categories)
and motivic measures out of it.
We will prove various identities \emph{after} having applied a certain motivic measure,
and conjecture that the same identities hold \emph{before} taking the measure.

Throughout this section, we let $\field$ denote a field.
For simplicity we will assume from the beginning that it is
algebraically closed field and of characteristic zero,
but this is not needed for some of the results we will use.

\begin{definition}
  The \emph{Grothendieck ring of varieties} $\Kzero(\mathrm{Var}/\field)$ is
  the free abelian group generated by isomorphism classes of varieties of finite type over $\field$,
  modulo the cut-and-paste relation $[X] = [Y] + [X\setminus Y]$
  whenever~$Y\subset X$ is a closed subvariety.
  The ring structure on $\Kzero(\mathrm{Var}/\field)$ is defined by the cartesian product of varieties.
\end{definition}

We set $\lefschetz = [\mathbb{A}^1_\field] \in \Kzero(\mathrm{Var}/\field)$,
so that~$[\mathbb{A}^n_\field] = \lefschetz^n$,
and~$[\mathbb{P}^n] = 1 + \lefschetz + \cdots + \lefschetz^n$.

Relations among the classes in $\Kzero(\mathrm{Var}/\field)$ are often detected
via homomorphisms of $\Kzero(\mathrm{Var}/\field)$ into other rings.
We will consider the following diagram of motivic measures
\begin{equation}
  \label{equation:diagram}
  \begin{tikzcd}
    \mathbb{Z}[\mathbb{L}] \arrow[d, "\mathbb{L}\mapsto 1"] \arrow[r, hook] & \Kzero(\mathrm{Var}/\field) \arrow[r, "\mu_\hodge"] \arrow[d, "\mu_{\derived^\bounded}"] & \Kzero(\mathrm{HS}) \arrow[r, "\mu_{\mathrm{E}}"] & \mathbb{Z}[x,y] \arrow[d] \arrow[r]                   & \mathbb{Z}[z] \arrow[d, "\chi"] \\
    \mathbb{Z} \arrow[r, hook]                                              & \Kzero(\mathrm{Cat}/\field) \arrow[rr, "\mu_\hochschild"]                                  &                                                   & \mathbb{Z}[t,t^{-1}] \arrow[r, dashed, swap, "(-)_0"] & \mathbb{Z},
  \end{tikzcd}
\end{equation}
where the subdiagram on the full arrows commutes,
and the subdiagram on the dashed arrows commutes on the image of~$\mathbb{Z}[\mathbb{L}]$,
see \cref{remark:comment-on-euler-characteristic}.
We will now recall the constructions of these measures,
and explain how they are related.
Whenever Hodge structures are involved, we will tacitly assume that~$\field=\mathbb{C}$.

Every variety over $\field$ is birational, via a resolution of singularities, to a nonsingular, projective variety, by our choice of~$\field$.
It follows that the natural map from the free abelian group generated by smooth, projective varieties over $\field$
to $\Kzero(\mathrm{Var}/\field)$ is surjective.
A theorem of Bittner \cite[Theorem~3.1]{MR2059227} asserts that the kernel of this map
is generated by~$[\emptyset] = 0$, and the relations~$([\mathrm{Bl}_Z(X)] - [E]) - ([X] - [Z])$,
where~$E$ is the exceptional divisor of the blowup~$\mathrm{Bl}_Z(X) \rightarrow X$ of~$X$ in~$Z$,
where both $X$ and $Z$ are smooth and projective.

We apply Bittner's presentation to construct the first motivic measure that appears in our conjecture.

\begin{example}[Hodge motivic measure]
  Consider the category~$\mathrm{HS}$ of polarizable pure rational Hodge structures.
  Using Bittner's presentation we have a well-defined motivic measure~$\mu_\hodge$,
  landing in the Grothendieck ring~$\Kzero(\mathrm{HS})$,
  by setting
  \begin{equation}
    [X]\mapsto[\HH^\bullet(X_\mathbb{C},\mathbb{Q})]=\bigoplus_{i=0}^{2\dim X}[\HH^i(X_\mathbb{C},\mathbb{Q})]
  \end{equation}
  whenever~$X$ is a smooth, projective variety.
\end{example}

To complete the top row of \eqref{equation:diagram} we discuss the following motivic measures.

\begin{example}[E-polynomial]
  \label{example:E-polynomial}
  The Hodge motivic measure
  factors the E-polynomial motivic measure~$\mu_{\mathrm{E}}\colon\Kzero(\mathrm{Var}/\field)\to\mathbb{Z}[x,y]$
  which is defined as
  \begin{equation}
    \label{equation:hodge-polynomial-measure}
    [X]\mapsto\sum_{p,q=0}^{\dim X}(-1)^{p+q}\mathrm{h}^{p,q}(X)x^py^q
  \end{equation}
  whenever~$X$ is a smooth, projective variety.
  By abuse of notation we will denote by~$\mu_{\mathrm{E}}$
  both the motivic measure just defined,
  and the morphism~$\Kzero(\mathrm{HS})\to\mathbb{Z}[x,y]$ defined by setting
  \begin{equation}
    [H]\mapsto\sum_{p,q\in\mathbb{Z}}(-1)^{p+q}\mathrm{h}^{p,q}(H)x^py^q
  \end{equation}
  whenever~$H$ is a pure Hodge structure,
  through which it factors.
\end{example}

\begin{example}[Betti numbers and Euler characteristic]
  By the Hodge decomposition we can consider the ring morphism
  \begin{equation}
    \mathbb{Z}[x,y]\mapsto\mathbb{Z}[z]:x,y\mapsto z.
  \end{equation}
  The composition with~$\mu_{\mathrm{E}}$ gives the Poincar\'e--Betti polynomial
  whenever~$X$ is a smooth projective variety.
  By further setting~$z=1$ we obtain the Euler characteristic of~$X$.
\end{example}

On the second row of \eqref{equation:diagram} we consider the Grothendieck ring of categories,
as introduced in \cite{MR2051435}.

\begin{definition}
  The \emph{Grothendieck ring of categories} $\Kzero(\mathrm{Cat}/\field)$ is
  the free abelian group generated by quasi-equivalence classes of smooth and proper pretriangulated dg categories
  modulo the ``cut-and-paste relation'' $[\mathcal{A}] = [\mathcal{B}] + [\mathcal{C}]$,
  whenever~$\mathrm{H}^0(\mathcal{A})=\langle\mathrm{H}^0(\mathcal{B}),\mathrm{H}^0(\mathcal{C})\rangle$
  is a semiorthogonal decomposition.
  The ring structure on $\Kzero(\mathrm{Cat}/\field)$ is defined by the tensor product of dg categories.
\end{definition}

\begin{example}[Derived categories]
  Using Bittner's presentation,
  and Orlov's blowup formula,
  we have a well-defined motivic measure~$\mu_{\derived^\bounded}$,
  landing in~$\Kzero(\mathrm{Cat}/\field)$,
  by setting
  \begin{equation}
    \label{equation:derived-category-measure}
    [X]\mapsto[\derived^\bounded(X)],
  \end{equation}
  whenever~$X$ is a smooth, projective variety.
  The left square in \eqref{equation:diagram} commutes.
\end{example}

\begin{example}[Hochschild homology polynomial]
  The motivic measure~$\mu_{\derived^\bounded}$
  factors the motivic measure~$\mu_\hochschild\colon\Kzero(\mathrm{Var}/\field)\to\mathbb{Z}[t,t^{-1}]$,
  which is defined as
  \begin{equation}
    [X]\mapsto\sum_{i=-\dim X}^{\dim X}\dim_\field\HHHH_i(X)t^i
  \end{equation}
  whenever~$X$ is a smooth, projective variety.

  The Hochschild homology of~$X$ is defined
  as~$\HHHH_i(X)\colonequals\Ext_{X\times X}^i(\Delta_*\mathcal{O}_X,\Delta_*\omega_X[\dim X])$.
  Hochschild homology is an additive invariant \cite[Theorem~1.5(c)]{MR1667558},
  sending semiorthogonal decompositions to direct sums,
  thus explaining why it factors through~$\Kzero(\mathrm{Cat}/\field)$.
  Again, by abuse of notation we also denote by~$\mu_\hochschild$
  the morphism~$\Kzero(\mathrm{Cat}/\field) \to \mathbb{Z}[t,t^{-1}]$.

  By the Hochschild--Kostant--Rosenberg decomposition,
  e.g., as in \cite{MR1390671},
  we have
  \begin{equation}
    \label{equation:hochschild-kostant-rosenberg}
    \HHHH_i(X)\cong\bigoplus_{q-p=i}\HH^p(X,\Omega_X^q),
  \end{equation}
  so the measure~$\mu_\hochschild$
  (or more precisely, the composition~$\mu_{\derived^\bounded}\circ\mu_\hochschild$)
  factors through~$\mu_\hodge$ and~$\mu_{\mathrm{E}}$,
  i.e.,
  the middle square in \eqref{equation:diagram} commutes.
\end{example}

\begin{remark}
  \label{remark:comment-on-euler-characteristic}
  We can consider both the Euler characteristic of~$X$,
  and the Euler characteristic of the Hochschild homology of~$X$.
  By \eqref{equation:hochschild-kostant-rosenberg},
  if all Hodge numbers~$\mathrm{h}^{p,q}(X)$ where~$p+q$ is odd vanish,
  then both Euler characteristics are computed by summing \emph{positive} integers,
  and thus they will agree.
  This is in particular the case if~$[X]\in\mathbb{Z}[\lefschetz]$.
\end{remark}

\begin{remark}
  \label{remark:topological-k-theory}
  It is possible to modify \eqref{equation:diagram} and add an analogue of~$\Kzero(\mathrm{HS})$ to its bottom row.
  To do so,
  we have to replace~$\Kzero(\mathrm{Cat}/\field)$
  by its subring~$\Kzero(\mathrm{geom\,Cat}/\field)$ given by geometric dg categories,
  i.e.,
  those dg categories which arise as admissible subcategories of derived categories of smooth projective varieties.
  The question whether every smooth and proper dg category is geometric,
  and thus whether the subring is all of~$\Kzero(\mathrm{Cat}/\field)$
  is raised in \cite{MR3545926}
  and remains open.

  By \cite[Proposition~5.4(1)]{MR4406785},
  the topological K-theory of a geometric dg
  comes equipped with a Hodge structure,
  for which
  \begin{equation}
    \operatorname{gr}^p(\operatorname{K}_i^{\mathrm{top}}(\mathcal{A})\otimes\mathbb{C})
    \cong
    \HHHH_{i+2p}(\mathcal{A}).
  \end{equation}
  By Bott periodicity, it will suffice for us to consider~$i=0,1$.
  This makes it possible to define the ring morphism
  \begin{equation}
    \Kzero(\mathrm{geom\,Cat}/\field)\to\Kzero(\mathbb{Z}/2\mathbb{Z}\mhyphen\mathrm{HS})
    :
    [\mathcal{A}]\mapsto[\operatorname{K}_0^{\mathrm{top}}(\mathcal{A})\oplus\operatorname{K}_1^{\mathrm{top}}(\mathcal{A})]
  \end{equation}
  landing in the Grothendieck ring of~$\mathbb{Z}/2\mathbb{Z}$-graded Hodge structures,
  and which factors the morphism~$\mu_\hochschild$ in the bottom row of \eqref{equation:diagram}.
  We will not use this in what follows.
\end{remark}

\section{Decomposing Fano schemes in the hyperelliptic case}
\label{section:hyperelliptic-decomposition}
Let us first recall a precursor to the web of conjectures and results regarding~$\fano_k(Q_1\cap Q_2)$.

By \cite[Theorem 1]{MR0429897}
there is an isomorphism~$\fano_{g-2}(Q_1\cap Q_2)\cong\mathrm{M}_C(2,\mathcal L)$,
where~$\mathrm{M}_C(2,\mathcal L)$ is the moduli space of stable rank-2 bundles
with fixed determinant~$\mathcal L$ of odd degree.
In \cite[Conjecture A]{MR4557892},
the authors conjectured and gave evidence for
a semiorthogonal decomposition for~$\derived^\bounded(\mathrm{M}_C(2,\mathcal L))$
in terms of~$C$ and its symmetric powers,
similar to \cref{conjecture:hyperelliptic}.
This decomposition has since been established:
\begin{itemize}
  \item \cite[Theorem 1.1]{2108.11951v5} gives the decomposition along with, possibly, a phantom component,
  \item \cite[Theorem~1.1]{2304.01825v2} shows the vanishing of the phantom component.
\end{itemize}
In fact, the conjecture and its proof work for an arbitrary smooth curve of genus~$g$,
not just hyperelliptic curves.
For~$k\neq g-2,g-1$ the Fano scheme does not have an interpretation
which also makes sense for non-hyperelliptic curves.

Thus, the context for \cref{conjecture:hyperelliptic}
is that it is an interpolation between semiorthogonal decompositions for~$Q_1\cap Q_2$ as in \eqref{equation:bondal-orlov-kuznetsov-sod} for~$k=0$
and~$\mathrm{M}_C(2,\mathcal{L})$ as in \cite[Theorem~1.1]{2304.01825v2} for~$k=g-2$.

As evidence for the (no longer conjectural) decomposition of~$\derived^\bounded(\fano_{g-2}(Q_1\cap Q_2))$,
it is shown in \cite[Theorem C]{MR4557892} that in~$\Kzero(\mathrm{Var}/\field)$
we have the identity
\begin{equation}
  \label{equation:BGM-K0}
  [\fano_{g-2}(Q_1\cap Q_2)]=\mathbb{L}^{g-1}[\Sym^{g-1}C]+\sum_{i=1}^{g-2}(\mathbb{L}^i +\mathbb{L}^{3g-3-2i})[\Sym^iC]+T,
\end{equation}
for some class~$T$ such that~$(1+\mathbb L)\cdot T=0$,
which is conjectured in op.~cit.~to vanish.
Thus, \cref{conjecture:hyperelliptic-refined} is the (conjectural) analogue of \eqref{equation:BGM-K0},
for different values of~$k$.

In this section we will prove that a decomposition result from \cite{MR4216588}
holds in the more refined setting involving Hodge structures,
and use this as the main evidence for \cref{conjecture:hyperelliptic-refined}.
As explained in \cref{subsection:hyperelliptic-evidence}
we obtain evidence for \cref{conjecture:hyperelliptic} from \cref{conjecture:hyperelliptic-refined}
by setting~$\lefschetz=1$.

\subsection{Preliminaries on mixed Hodge modules}
\label{subsection:mhm}
We recall some basics on mixed Hodge modules
which are necessary for the proof of \cref{theorem:hodge-decomposition}.
The reader who is familiar with the theory
can skip ahead to \cref{subsection:upgrade}.

\paragraph{Hodge modules}
Let $X$ be a smooth variety over $\mathbb{C}$.
We freely use the language of mixed Hodge modules on~$X$ due to M.~Saito;
the basic formal features of the theory are summarized, for instance, in~\cite{MR1042805}.
Fortunately, we only require a few general properties of the theory, which we now review.
\begin{enumerate}
  \item There is an abelian category~$\MHM(X)$ of mixed Hodge modules on~$X$,
    whose derived category~$\derived^\bounded(\MHM(X))$
    enriches the derived category of constructible sheaves
    $\derived^\bounded_{\mathrm{c}}(X, \mathbb{Q})$,
    in the sense that there is a \emph{rationalization} functor
    \begin{equation}
      \mathrm{rat}_X\colon\derived^\bounded(\MHM(X)) \to \derived^\bounded_{\mathrm{c}}(X, \mathbb{Q}).
    \end{equation}
    Given a morphism $f\colon X \to Y$,
    the familiar operations on derived categories of constructible sheaves
    (e.g., $f_*, f^*, f^!, f_!$) lift to $\derived^\bounded(\MHM(X))$.
  \item The abelian category $\MHM(\mathrm{pt})$
    is equivalent to the category of~$\mathbb{Q}$-mixed Hodge structures.
  \item Each mixed Hodge module has an underlying D-module.
    Moreover, there is a forgetful functor from~$\derived^\bounded(\MHM(X))$
    to~$\derived^\bounded_{\mathrm{coh}}(\mathrm{D}_X)$,
    where $\mathrm{D}_X$ is the ring of differential operators on $X$.
    By convention, we always work with left D-modules.
\end{enumerate}

\begin{example}
  Let $X$ be a smooth variety over $\mathbb{C}$ of dimension $d$.
  We write $\mathbb{Q}^\rH[d]$ for the constant Hodge module,
  whose underlying perverse sheaf is the shifted constant sheaf $\mathbb{Q}[d]$.
\end{example}

\begin{example}[IC-extensions]
  Let $X$ be a smooth variety over $\mathbb{C}$.
  Let $Z \subset X$ be an integral subvariety of $X$,
  and let $U \subset Z$ be a nonempty open subvariety.
  Given a variation of Hodge structure $\mathcal{V}$ on $U$,
  a fundamental result of Saito \cite{MR1047415}
  says that $\mathcal{V}$ extends in a unique and functorial manner to
  a mixed (in fact, \emph{pure}) Hodge module on $X$,
  written $\IC(Z, \mathcal{V})$, with support on $Z$.
  As the notation suggests,
  the rationalization of the Hodge module $\IC(Z, V)$ is $i_* j_{!*} \mathcal{V}$,
  where $i\colon Z \to X$, $j\colon U \to Z$, and $j_{!*}$ is the intermediate extension.
\end{example}

\begin{example}[Strict support decomposition]
  \label{example:strict-support}
  The previous example is universal in the following sense:
  Let $M$ be a pure Hodge module. Then there is a decomposition
  \begin{equation}
    M \cong \bigoplus_i \IC(Z_i, \mathcal{V}_i),
  \end{equation}
  where $\{Z_i\}$ is a finite set of irreducible closed subvarieties of $X$,
  and $\mathcal{V}_i$ is a variation of Hodge structure on an open subvariety of $Z_i$.
  The reader may take this to be the definition of a pure Hodge module, although we have refrained from discussing weights.
\end{example}

\paragraph{Fourier transforms}
A sheaf of~$\mathbb{Q}$-vector spaces on~$\mathbb{C}^n \times X$
is said to be \emph{monodromic} if
its restriction to each $\mathbb{G}_{\mathrm{m}}$-orbit
(for the scaling action of~$\mathbb{G}_{\mathrm{m}}$ on $\mathbb{C}^n$)
is locally constant.
An object of~$\derived^\bounded_{\mathrm{c}}(\mathbb{C}^n \times X, \mathbb{Q})$
is monodromic if its cohomology sheaves are monodromic.
Finally, $M \in \derived^\bounded(\MHM(X))$ is monodromic if its rationalization is monodromic.

\begin{remark}
  \label{remark:monodromic-d-module}
  There is a corresponding notion for a D-module to be monodromic,
  which means (roughly) that it admits a generalized eigenspace decomposition with respect to the Euler operator $\sum x_i \partial_{x_i}$.
  If $M$ is a regular holonomic D-module,
  then $M$ is monodromic (in this sense) if and only if
  the corresponding perverse sheaf is monodromic in the sense described above,
  cf.~\cite[Proposition~7.12]{MR0864073}.
\end{remark}

The \emph{Fourier--Sato transform} carries monodromic perverse sheaves on~$\mathbb{C}^n \times X$
to monodromic perverse sheaves on~$\mathbb{C}^{n, \vee} \times X$.
For example, in the case~$X = \mathrm{pt}$,
the skyscraper sheaf~$\mathbb{Q}_0$ supported at~$0 \in \mathbb{C}^n$
is carried to the perverse sheaf~$\mathbb{Q}[n]$.
A comprehensive treatment may be found in \cite{MR1299726},
but \cref{example:fourier-saito-of-subspace} is sufficient for the time being;
a general definition for Hodge modules will be given below.

There is a corresponding operation for D-modules,
often called the \emph{Fourier--Laplace transform},
which (in rough terms) exchanges the action of the differential operators~$x_1, \dots, x_n$
(resp.,~$\partial_{x_1}, \dots, \partial_{x_n}$) pulled back from~$\mathbb{C}^n$
with the action of~$\partial_{y_1}, \dots, \partial_{y_n}$
(resp.,~$y_1, \dots, y_n$) pulled back from~$\mathbb{C}^{n, \vee}$.
Here,~$y_1, \dots, y_n$ is the basis dual to~$x_1, \dots, x_n$.

\begin{example}
  \label{example:fourier-saito-of-subspace}
  Let $\iota\colon V \subset \mathbb{C}^n$ be a vector subspace of dimension~$d$,
  and let $\iota^\perp\colon V^{\perp} \subset \mathbb{C}^{n, \vee}$ be the annihilator of~$V$,
  i.e., $V^{\perp} = \{w \mid w(V) = 0\}$.
  Then
  \begin{equation}
    \FS(\iota_* \mathbb{Q}[d]) \cong \iota^\perp_* \mathbb{Q}[d^\perp],
  \end{equation}
  where~$d^\perp = \dim V^\perp = n - d$.

  To prove this, one may write~$\mathbb{C}^n \cong V \oplus W$ for some subspace~$W$.

  There is a natural isomorphism~$\mathbb{C}^{n, \vee} \to V^\vee \oplus W^\vee$
  which induces an isomorphism from~$V^\perp$ to~$W^\vee$.
  We write
  \begin{equation}
    \iota_* \mathbb{Q}[d] \cong \mathbb{Q}[d] \boxtimes \mathbb{Q}_0,
  \end{equation}
  where~$\mathbb{Q}_0$ is the skyscraper sheaf supported at~$0 \in W$.

  There is a compatibility between exterior products and the Fourier--Sato transformation \cite[Proposition~3.7.15]{MR1299726}
  (which uses different terminology).
  In our case, this means that
  \begin{equation}
    \FS(\mathbb{Q}[d] \boxtimes \mathbb{Q}_0)
    \cong \FS(\mathbb{Q}[d]) \boxtimes \FS(\mathbb{Q}_0)
    \cong \mathbb{Q}_0 \boxtimes \mathbb{Q}[\dim W^\vee].
  \end{equation}
  Under the isomorphism between~$\mathbb{C}^{n, \vee}$ and~$V^\vee \oplus W^{\vee}$,
  the rightmost term is identified with~$\iota^\perp_* \mathbb{Q}[d^\perp]$.
\end{example}

\paragraph{Monodromic Hodge modules}
We now describe the Fourier--Laplace transform for monodromic mixed Hodge modules,
first introduced by T. Saito \cite{MR4434749}.
In fact, there are many possible mixed Hodge module structures
on the Fourier--Laplace transform of the underlying D-module, see~\cite[Remark~3.24]{MR4434749}.
We have adopted the ``geometric'' approach of \cite{MR4610959}.

Suppose that~$M$ is a mixed Hodge module on~$\mathbb{C}^n \times X$.
Consider the diagram
\begin{equation}
  \label{equation:fourier-diagram}
  \begin{tikzcd}
    \mathbb{C}^n \times \mathbb{C}^{n, \vee} \times X \ar[d, bend left=15, "p^\vee"] \ar[r, "p"] & \mathbb{C}^n \times X \\
    \mathbb{C}^{n, \vee} \times X, \ar[u, bend left=15, "\sigma"]
  \end{tikzcd}
\end{equation}
where~$\sigma$ is the zero-section, and~$p, p^\vee$ are the projections.
Finally, there is a function $g\colon\mathbb{C}^n \times \mathbb{C}^{n, \vee} \times X \to \mathbb{C}$,
given by the natural pairing between the fibers of~$\mathbb{C}^n$ and~$\mathbb{C}^{n,\vee}$ over~$X$.
We recall the following definition from \cite{MR4610959}.

\begin{definition}
  \label{definition:fourier-laplace-transform}
  Let $M$ be a mixed Hodge module on $\mathbb{C}^n \times X$.
  Then we define
  \begin{equation}
    \FL(M) = \sigma^* \phi_g p^! M,
  \end{equation}
  where $\phi_g$ denotes the vanishing cycles functor along $g$.
\end{definition}

\begin{remark}
\label{remark:all-functors-derived}
  When working with mixed Hodge modules, our notational convention is that all functors (e.g., $f_*, f^*, f_!, f^!$) are derived.
\end{remark}

\begin{remark}
  In \cite{MR4610959}, the stated definition is $\mathcal{H}^0 \sigma^* \phi_g p^! M[-n]$,
  where cohomology is taken with respect to the standard t-structure on D-modules,
  which corresponds to the perverse t-structure on $\derived^\bounded_{\mathrm{c}}(X, \mathbb{Q})$ under the Riemann--Hilbert correspondence.
  (The shift by $-n$, which we have not followed, is a matter of convention.)
  In fact, we claim that the rationalization of~$\sigma^* \phi_g p^! M$ is concentrated in a perverse single degree,
  so the difference between $\FL(M)$ and $\mathcal{H}^{-n}\FL(M)$ is a formality.

  Indeed, after rationalization and commuting vanishing cycles with proper pushforward
  (as in the proof of \cref{lemma:fourier-laplace-commutes-with-pushforward} below),
  the formula above for $\FL(M)$ coincides with formula (10.3.31) of \cite{MR1299726}
  for the Fourier--Sato transform of the rationalization of~$M$.
  Then \cite[Proposition~10.3.18]{MR1299726} shows that the Fourier--Sato transform is perverse t-exact.
\end{remark}

\begin{lemma}
  \label{lemma:fourier-laplace-commutes-with-pushforward}
  Let~$X$ be a smooth, proper variety,
  and let~$M \in \derived^\bounded(\MHM(X))$ be monodromic.
  Suppose that~$f\colon X \to Y$ is a proper morphism, and let
  \begin{equation}
    \begin{gathered}
      \pi = 1 \times f\colon\mathbb{C}^n \times X \to \mathbb{C}^n \times Y \\
      \pi^\vee = 1 \times f\colon\mathbb{C}^{n, \vee} \times X \to \mathbb{C}^{n, \vee} \times Y
    \end{gathered}
  \end{equation}
  be the corresponding maps.
  Then~$\pi_* M$ is a monodromic mixed Hodge module,
  and there is an isomorphism of mixed Hodge modules
  \begin{equation}
    \FL(\pi_* M) \cong \pi^\vee_* \FL(M).
  \end{equation}
\end{lemma}

\begin{proof}
    The fact that monodromicity of constructible complexes
    is preserved by pushforward is given by \cite[Proposition~8.5.7]{MR1299726}.
    We apply this fact to the rationalization of~$M$.

    For the identity, the idea is to commute~$\pi_*$ with each of the functors~$p^!, \phi_g$, and~$\sigma_*$ used to define~$\FL(M)$.
    We write
    \begin{equation}
      \begin{tikzcd}
        \mathbb{C}^n \times \mathbb{C}^{n, \vee} \times Y \ar[d, bend left=15, "q^\vee"] \ar[r, "q"] & \mathbb{C}^n \times Y \\
        \mathbb{C}^{n, \vee} \times Y \ar[u, "\tau", bend left=15]
      \end{tikzcd}
    \end{equation}
    for the analogue of \eqref{equation:fourier-diagram},
    and let~$h\colon\mathbb{C}^n \times \mathbb{C}^{n, \vee} \times Y \to \mathbb{C} \times Y$ be the natural pairing.

    There is a cartesian square
    \begin{equation}
      \begin{tikzcd}
        \mathbb{C}^{n,\vee} \times X \ar[r, "\sigma"] \ar[d, "\pi^\vee"] & \mathbb{C}^n \times \mathbb{C}^{n, \vee} \times X \ar[d, "\pi \times \pi^\vee"] \\
        \mathbb{C}^{n, \vee} \times Y \ar[r, "\tau"] & \mathbb{C}^{n} \times \mathbb{C}^{n, \vee} \times Y.
      \end{tikzcd}
    \end{equation}
    Since the vertical maps are proper,
    by proper base change for mixed Hodge modules
    \cite[\S4.4.3]{MR1047415}
    there is an isomorphism of functors
    \begin{equation}
      \pi^\vee_* \circ \sigma^* \cong \tau^*\circ (\pi \times \pi^\vee)_*.
    \end{equation}
    Next, $g = h \circ (\pi \times \pi^\vee)$,
    and according to the compatibility between vanishing cycles
    and proper pushforward \cite[Theorem~2.14]{MR1047415}
    there is an isomorphism of functors
    \begin{equation}
      (\pi \times \pi^\vee)_* \circ \phi_g \cong \phi_h \circ (\pi \times \pi^\vee)_*.
    \end{equation}
    Finally, applying proper base change once more
    to the the cartesian square
    \begin{equation}
      \begin{tikzcd}
        \mathbb{C}^n \times \mathbb{C}^{n,\vee} \times X \ar[r, "p"] \ar[d, "\pi \times \pi^\vee"] & \mathbb{C}^n \times X \ar[d, "\pi"] \\
        \mathbb{C}^n \times \mathbb{C}^{n, \vee} \times Y \ar[r, "q"] & \mathbb{C}^n \times Y,
      \end{tikzcd}
    \end{equation}
    we obtain an isomorphism of functors~$q^! \circ \pi_* \cong (\pi \times \pi^\vee)_* \circ p^!$.
    The statement follows from putting all of the identities together.
\end{proof}

\begin{lemma}
  \label{lemma:orthogonal-in-general}
  Let~$X$ be a smooth variety over~$\mathbb{C}$.
  Let~$\iota\colon E \subset \mathbb{C}^n \times X$ be a vector bundle over~$X$,
  and let~$\iota^\perp\colon E^{\perp} \subset \mathbb{C}^{n, \vee} \times X$
  be the annihilator of~$E$.
  We write $d = \dim E$, $d^{\perp} = \dim E^{\perp}$.
  \begin{enumerate}
    \item\label{enumerate:orthogonal-in-general-1}
      The Hodge module $\iota_* \mathbb{Q}^\rH[d]$ is monodromic.
    \item\label{enumerate:orthogonal-in-general-2}
      There is an isomorphism of Hodge modules
      \begin{equation}
        \FL(\iota_*\mathbb{Q}^\rH[d]) \cong \iota^{\perp}_* \mathbb{Q}^\rH[d^{\perp}](d - d^{\perp}).
      \end{equation}
    \item \label{enumerate:orthogonal-in-general-3}
      Let $\pi\colon\mathbb{C}^n \times X \to X$
      and $\pi^\vee\colon\mathbb{C}^{n, \vee} \times X \to X$ be the projections.
      Then~$\pi_* \iota_* \mathbb{Q}^\rH[d]$ is monodromic,
      and there is an isomorphism of mixed Hodge modules
      \begin{equation}
        \FL(\pi_* \iota_* \mathbb{Q}^\rH[d]) \cong \pi^\vee_* \iota^{\perp}_* \mathbb{Q}^\rH[d^\perp](d - d^\perp)
      \end{equation}
   \end{enumerate}
\end{lemma}

\begin{proof}
  \Cref{enumerate:orthogonal-in-general-1} is a direct consequence of the definition:
  the underlying perverse sheaf of~$\iota_* \mathbb{Q}^\rH[d]$ is~$\iota_* \mathbb{Q}[d]$,
  and the only nonzero cohomology sheaf is constant along each~$\mathbb{G}_{\mathrm{m}}$-orbit in $\mathbb{C}^n \times X$.

  For \cref{enumerate:orthogonal-in-general-2},
  we argue indirectly as follows:
  From \cref{example:fourier-saito-of-subspace}, we know that the statement is true
  for the underlying perverse sheaves, at least locally over $X$.
  On the other hand, from the existence of the strict support decomposition
  described in \cref{example:strict-support},
  both sides are IC-extensions of variations of Hodge structures on $E^{\perp}$
  whose underlying local systems have rank $1$ and trivial monodromy
  (since the property of having trivial monodromy may be checked on a Zariski-open cover).
  It follows that the variations are isomorphic up to a twist.

  For \cref{enumerate:orthogonal-in-general-3},
  the fact that the pushforward is monodromic is given in \cref{lemma:fourier-laplace-commutes-with-pushforward},
  which also furnishes the identity
  \begin{equation}
    \FL(\pi_* \iota_* \mathbb{Q}^\rH[d]) \cong \pi^\vee_* \FL(\iota_* \mathbb{Q}^\rH[d])).
  \end{equation}
  The right-hand side is~$\pi_*^\vee \iota_*^\perp \mathbb{Q}^\rH[d^\perp](d - d^\perp)$,
  as one may see by pushing forward \cref{enumerate:orthogonal-in-general-2}.
\end{proof}

\subsection{Cohomology of the Fano scheme}
\label{subsection:upgrade}
The cohomology of~$\fano_k(Q_1 \cap Q_2)$ is computed in~\cite[Theorem 1.1]{MR3689749},
where an isomorphism of vector spaces
\begin{equation}
  \label{equation:chen-vilonen-xue-hyperelliptic}
  \HH^{d-i}(\fano_k(Q_1 \cap Q_2),\mathbb{C})
  \cong
  \bigoplus_{j=d-k-1}^g \left( \bigwedge\nolimits^{g-j}\HH^1(C,\mathbb{C}) \right)^{\oplus N(g-k,j;d-i)}
\end{equation}
is given.
Here~$N(a,b;c)$ denotes the coefficient of~$q^c$ in
\begin{equation}
  \label{equation:coefficient-expression}
  q^{-(b-a+1)(2a-1)}(1-q^{4b})\frac{\prod_{\ell=b-a+2}^{a+b-2}(1-q^{2\ell})}{\prod_{\ell=1}^{2a-2} (1-q^{2\ell})}.
\end{equation}
This result does not deal with the Hodge structures
on both sides of \eqref{equation:chen-vilonen-xue-hyperelliptic}.
We will upgrade the proof of \cite[Theorem 1.1]{MR3689749},
and explain the necessary modifications to obtain \cref{theorem:hodge-decomposition},
which upgrades \eqref{equation:chen-vilonen-xue-hyperelliptic}
to an isomorphism of Hodge structures.
From now on,
until the end of \cref{subsection:upgrade},
we will use the notation of \cite{MR3689749}.

\paragraph{Springer theory}
Let~$V$ be a~$\mathbb{C}$-vector space of dimension~$2n$.
We write~$G = \mathrm{SL}(V)$, and~$K = \mathrm{SO}(V, q)$,
where~$q$ is a nondegenerate quadratic form on~$V$.
We write~$\mathfrak{g}$ for the Lie algebra of~$G$,
and~$\mathfrak{g}_1$ for the Lie algebra of~$K$, respectively;
similarly, we write~$\mathcal{N}$
and~$\mathcal{N}_1$ for the nilpotent cones of~$G$ and~$K$, respectively.

We write~$\mathfrak{g}^{\rs}$ for the regular semisimple elements of~$\mathfrak{g}$,
and set~$\mathfrak{g}^{\rs}_1\colonequals\mathfrak{g}_1 \cap \mathfrak{g}^{\rs}$.
To each element~$\gamma \in \mathfrak{g}_1^{\rs}$,
one may associate the nondegenerate quadratic form~$(\gamma -, -)$,
where~$(-,-)$ is the bilinear form corresponding to $q$.
Thus to translate back into our notation,
one takes~$2n=2g+2$,
but in what follows we will use the notation of op.~cit.,
and use affine (instead of projective) dimensions.

The proof of \cref{theorem:hodge-decomposition}
proceeds by analyzing the relationship between the following three constructions from \cite{MR3689749}:
\begin{enumerate}
  \item
    There is a morphism $\mathcal{C} \to \mathfrak{g}^\rs_1$
    whose fiber over $\gamma \in \mathfrak{g}^\rs_1$
    is the hyperelliptic curve~$C_\gamma$ of genus $g$
    given by the double cover with equation
    \begin{equation}
      y^2 = \det(t \cdot 1 - \gamma).
    \end{equation}
    We write $f\colon\Jac \to \mathfrak{g}^\rs_1$ for the associated Jacobian fibration.

  \item
    For each $i=0,\ldots,g$
    we consider the family
    \begin{equation}
      v_i\colon E_i = \left\{(\gamma \in \mathfrak{g}_1, 0 \subset H_i \subset H_i^{\perp} \subset \mathbb{C}^{2n}) \mid  \gamma(H^{\perp}) = 0\right\} \to \mathfrak{g}_1
    \end{equation}
    where~$\dim H_i=i$.
    This is Reeder's resolution of the nilpotent orbit~$\overline{\mathrm{O}}_{2^i 1^{2n - 2i}}$
    from \cite{MR1328333}.

  \item
    For each $i=0,\ldots,g$
    we consider the family
    \begin{equation}
      \widecheck{v}_i\colon\widecheck{E}_i = \left\{(\widecheck{\gamma} \in \mathfrak{g}_1, 0 \subset H_i \subset H_i^{\perp} \subset \mathbb{C}^{2n}) \mid \widecheck{\gamma}(H_i) \subset H_i^{\perp} \right\} \to \mathfrak{g}_1
    \end{equation}
    where~$\dim H_i=i$.
    Over a regular semisimple element~$\gamma'$
    corresponding to a smooth quadric $Q'$,
    the fiber $\widecheck{E}_{i,\gamma'}$
    is the Fano scheme of (affine) $i$-planes
    on the intersection $Q' \cap Q$,
    where $Q$ is the smooth quadric associated to the fixed quadratic form $q$.
\end{enumerate}

We regard $E_i$ and $\widecheck{E}_i$ as subvarieties of $\mathfrak{g}_1 \times F_i$,
where~$F_i = \{0 \subset H_i \subset H^{\perp}_i \subset \mathbb{C}^{2n}\}$
is the orthogonal Grassmannian~$\OGr(i,q)=\OGr(i,2n)$.

\begin{lemma}
  \label{lemma:orthogonal-in-proof}
  With the above setup, we have that:
  \begin{enumerate}
    \item\label{enumerate:orthogonal-in-proof-1} $E_i$ and $\widecheck{E}_i$ are vector bundles on~$F_i$.
    \item\label{enumerate:orthogonal-in-proof-2} With respect to the Killing form on the fibers of~$\mathfrak{g}_1 \times F_i \to F_i$,
      one has $\widecheck{E}_i = (E_i)^{\perp}$.
  \end{enumerate}
\end{lemma}

\begin{proof}
  For \cref{enumerate:orthogonal-in-proof-1},
  the fibers of $E_i$ and $\widecheck{E}_i$ over $F_i$
  are vector subspaces of $\mathfrak{g}_1$;
  to see that the rank is constant, one considers the action of $\mathrm{O}(V, q)$ on $F_i$, $E_i$, and $\widecheck{E}_i$.

  For \cref{enumerate:orthogonal-in-proof-2},
  again using the action of $\mathrm{O}(V, q)$,
  one may reduce to the case of the fiber over a standard isotropic subspace.
  Up to scaling, the Killing form coincides with the trace form $M, N \mapsto \mathrm{Tr}(MN)$,
  and the result is a direct calculation.
\end{proof}

\paragraph{Upgrading to mixed Hodge modules}
The Killing form induces an isomorphism between $\mathfrak{g}_1$ and~$\mathfrak{g}_1^{\vee}$,
so we may regard the Fourier--Laplace transform as taking monodromic Hodge modules on $\mathfrak{g}_1$
to monodromic Hodge modules on $\mathfrak{g}_1$.

The following duality result is the first step in upgrading
the proof of \cite[Theorem 1.1]{MR3689749}.

\begin{proposition}
  \label{proposition:duality}
  Let $d_i = \dim E_i$, $\widecheck{d}_i = \dim \widecheck{E}_i$.
  \begin{enumerate}
    \item The Hodge modules $v_{i,*} \mathbb{Q}^\rH[d_i]$
      and $\widecheck{v}_{i,*} \mathbb{Q}^\rH[\widecheck{d}_i]$ are monodromic.

    \item There is an isomorphism
      \begin{equation}
        \label{equation:duality}
        \FL(v_{i,*} \mathbb{Q}^\rH[d_i]) \cong \widecheck{v}_{i,*} \mathbb{Q}^\rH[\widecheck{d}_i](d_i - \widecheck{d}_i)
      \end{equation}
      in~$\derived^\bounded(\MHM(\mathfrak{g}_1))$.
  \end{enumerate}
\end{proposition}

\begin{proof}
    Given \cref{lemma:orthogonal-in-proof}, both parts follow from \cref{lemma:orthogonal-in-general}.
\end{proof}

Given \cref{proposition:duality}, the remainder of the proof of \cref{theorem:hodge-decomposition} proceeds as follows.

The dimension and monodromy calculation in \cite[Lemma 2.1]{MR3689749}
shows that $v_i$ is semismall for each $i$
and that the decomposition theorem for $v_i$ takes the form
\begin{equation}
  \label{equation:IC}
  v_{i,*} \mathbb{Q}^\rH[d_i]
  \cong
  \bigoplus_{j=0}^i \bigoplus_{k=0}^{m_{i,j}} \IC(\overline{\mathrm{O}}_{2^j 1^{2n - 2j}}, \mathbb{Q}^\rH)^{\oplus s_{i,j,k}} [\pm k],
\end{equation}
in $\derived^\bounded(\MHM(\mathfrak{g}_1))$,
where $m_{i,j}$ and $s_{i,j,k}$ are certain integers.
Therefore, the next goal is to identify
the Fourier--Laplace transform of the IC-sheaves
appearing in \eqref{equation:IC}.

Consider the Jacobian fibration $f\colon\Jac \to \mathfrak{g}_1^{\rs}$,
and define the variation of Hodge structure
\begin{equation}
  W_j\colonequals (f_* \mathbb{Q}^\rH)_{\mathrm{prim}},
\end{equation}
where the subscript denotes the primitive part of the variation with respect to the relative theta divisor, and we emphasize that the pushforward is derived (\cref{remark:all-functors-derived}).

\begin{lemma}
  \label{lemma:fourier-laplace-nilpotent-orbit}
  There is an isomorphism
  \begin{equation}
    \FL(\IC(\overline{\mathrm{O}}_{2^j 1^{2n - 2j}}, \mathbb{Q}^{\mathrm{H}})) \cong \IC(\mathfrak{g}_1, W_j),
  \end{equation}
  in $\derived^\bounded(\MHM(\mathfrak{g}_1))$,
  where we have suppressed Tate twists.
\end{lemma}

\begin{proof}
  With \cref{proposition:duality} in hand,
  the argument of the proof of \cite[Proposition 2.3]{MR3689749}
  works \textit{mutatis mutandis} in the category of Hodge modules.
\end{proof}

\begin{proof}[Proof of \cref{theorem:hodge-decomposition}]
  Returning to \eqref{equation:IC}
  and applying the Fourier--Laplace transform and \cref{lemma:fourier-laplace-nilpotent-orbit},
  we obtain isomorphisms (with twists suppressed)
  \begin{equation}
    \begin{aligned}
      \widecheck v_{i,*} \mathbb{Q}^\rH[\widecheck d_i]
      &\cong \bigoplus_{j=0}^i \bigoplus_{k=0}^{m_{i,j}} \FL(\IC(\overline{\mathrm{O}}_{2^j 1^{2n - 2j}}, \mathbb{Q}^\rH))^{\oplus s_{i,j,k}}[\pm k] \\
      &\cong \bigoplus_{j=0}^i \bigoplus_{k=0}^{m_{i,j}} \IC(\mathfrak{g}_1, W_j)^{s_{i,j,k}}[\pm k].
    \end{aligned}
  \end{equation}
  in $\derived^\bounded(\MHM(\mathfrak{g}_1))$.
  From here, the conclusion of the proof proceeds exactly as in the proof of \cite[Theorem~1.1]{MR3689749},
  by passing to cohomology and the fiber at an element of $\mathfrak{g}_1^{\rs}$.
\end{proof}

\subsection{Evidence for \texorpdfstring{\cref{conjecture:hyperelliptic,conjecture:hyperelliptic-refined}}{Conjectures \ref{conjecture:hyperelliptic} and \ref{conjecture:hyperelliptic-refined}}}
\label{subsection:hyperelliptic-evidence}
As evidenced by \eqref{equation:BGM-K0},
the conjectural identity \eqref{equation:hyperelliptic-refined}
for~$k=g-2$
is known to hold up to a class~$T\in\Kzero(\Var/\field)$ for which~$(1+\lefschetz)\cdot T=0$.
Extending this type of identity from~$g-2$ to all~$k=0,\ldots,g-2$
is the topic of \cref{conjecture:hyperelliptic-refined}.

The following lemma explains how it holds for~$k=0$ up to multiplication with~$\lefschetz$.
\begin{lemma}
  \label{lemma:k=0-calculation}
  Let~$Q_1\cap Q_2\subset\mathbb{P}^{2g+1}$ be a smooth intersection of quadrics,
  and let~$C$ be the associated hyperelliptic curve.
  We have that
  \begin{equation}
    \lefschetz[Q_1\cap Q_2]
    =
    \lefschetz\left( [\mathbb{P}^{2g-1}]-\lefschetz^{g-1}[\mathbb{P}^1]+\lefschetz^{g-1}[C] \right)
  \end{equation}
  in~$\Kzero(\Var/\field)$.
\end{lemma}

\begin{proof}
  Consider the total space~$p\colon\mathcal{Q}\to\mathbb{P}^1$
  of the pencil spanned by~$Q_1$ and~$Q_2$.
  Because we are working over an algebraically closed field,
  \cite[Corollary~2.7]{MR3848025} gives us that
  \begin{equation}
    [\mathcal{Q}]=[\mathbb{P}^1][\mathbb{P}^{g-1}](1+\lefschetz^{g+1})+[C]\lefschetz^g,
  \end{equation}
  as the hyperelliptic curve~$C$ also arises as the hyperbolic reduction.
  On the other hand we have that~$\mathcal{Q}\cong\operatorname{Bl}_{Q_1\cap Q_2}\mathbb{P}^{2g+1}$,
  thus
  \begin{equation}
    [\mathcal{Q}]=[\mathbb{P}^{2g+1}]+\lefschetz[Q_1\cap Q_2].
  \end{equation}
  An elementary calculation shows that
  \begin{equation}
    [\mathbb{P}^1][\mathbb{P}^{g-1}](1+\lefschetz^{g+1})
    =
    [\mathbb{P}^{2g+1}]+\lefschetz[\mathbb{P}^{2g-1}]-\lefschetz^g[\mathbb{P}^1],
  \end{equation}
  which finishes the proof.
\end{proof}

\begin{remark}
  \label{remark:g=2-3-folds}
  \Cref{conjecture:hyperelliptic-refined} is also known to hold in the case of~$g=2$
  for classical reasons:
  the projection from a line~$L$ on~$Q_1\cap Q_2\subset\mathbb{P}^5$
  can be resolved by blowing up,
  which exhibits~$\operatorname{Bl}_LQ_1\cap Q_2$
  as~$\operatorname{Bl}_C\mathbb{P}^3$,
  giving the identity \eqref{equation:hyperelliptic-refined}
  for~$g=2$ and~$k=0$ without the factor present in \cref{lemma:k=0-calculation}.
\end{remark}

The cases~$k=0,g-2$ of \cref{conjecture:hyperelliptic} are known to hold
by the work of Bondal--Orlov, Kuznetsov and Tevelev--Torres,
as explained in the introduction.

\paragraph{Verifying decompositions after applying the E-polynomial motivic measure}
For fixed~$g\geq 2$ and~$k=0,\ldots,g-2$
one can use \cref{theorem:hodge-decomposition}
to explicitly compute the Hodge numbers of~$\fano_k(Q_1\cap Q_2)$.
The Hodge numbers of~$C$ are immediate,
and those of~$\Sym^iC$ follow from, e.g., \cite[Theorem~1.1]{MR2777820}.
The use of these building blocks in \cref{conjecture:hyperelliptic-refined}
is motivated by interpolating between~\eqref{equation:bondal-orlov-kuznetsov-sod}
and the BGMN conjecture \cite{MR4557892}.

In \cite{hodge-diamond-cutter} the Hodge numbers of~$\fano_k(Q_1\cap Q_2)$
are implemented,
following \cref{theorem:hodge-decomposition}.
The code in \cref{appendix:code}
thus verifies \cref{conjecture:hyperelliptic-refined}
for a given~$g$ and~$k$
\emph{after} applying the motivic measure~$\mu_{\mathrm{E}}$ from \cref{example:E-polynomial}.

Let us illustrate one example.
\begin{example}
  Let~$g=4$ and~$k=1$.
  The identity in \eqref{equation:hyperelliptic-refined}
  reads
  \begin{equation}
    [\fano_1(Q_1\cap Q_2)]
    =
    ([\mathrm{pt}]+\lefschetz+\lefschetz^2 + \lefschetz^4+\lefschetz^5+\lefschetz^6 + \lefschetz^8+\lefschetz^9+\lefschetz^{10})
    + (\lefschetz^2+\lefschetz^3+\lefschetz^6+\lefschetz^7)[C]
    + \lefschetz^4[\Sym^2].
  \end{equation}
  By taking the E-polynomial motivic measure,
  we can check the identity on the level of Hodge diamonds
  as in \cref{figure:hodge-diamond-decomposition}.
  \begin{figure}[ht]
    \centering
    \begin{equation}
      \begin{array}{ccccccccccccccccc}
          &   & 1                               &   &   &   & 1         &                       & \\
          &   &                                 &   &   &   &           &                       & \\
          &   & 1                               &   &   &   & 1         &                       & \\
          &   &                                 &   &   &   &           &                       & \\
          &   & 2                               &   &   &   & 1         &                       &    & 1 \\
          & 4 &                                 & 4 &   &   &           &                       & 4  &                         & 4 \\
          &   & 2                               &   &   &   &           &                       &    & 2\\
          & 4 &                                 & 4 &   &   &           &                       & 4  &                         & 4 \\
          &   & 3                               &   &   &   & 1         &                       &    & 1                       &      &                       &   &   & 1 \\
          & 4 &                                 & 4 &   &   &           &                       &    &                         &      &                       &   & 4 &                               & 4 \\
        6 &   & 18                              &   & 6 & = & 1         & {\ }\quad + \quad{\ } &    &                         &      & {\ }\quad + \quad{\ } & 6 &   & \mathclap{17}                 &      & 6 \\
          & 4 &                                 & 4 &   &   &           &                       &    &                         &      &                       &   & 4 &                               & 4 \\
          &   & 3                               &   &   &   & 1         &                       &    & 1                       &      &                       &   &   & 1 \\
          & 4 &                                 & 4 &   &   &           &                       & 4  &                         & 4 \\
          &   & 2                               &   &   &   &           &                       &    & 2\\
          & 4 &                                 & 4 &   &   &           &                       & 4  &                         & 4 \\
          &   & 2                               &   &   &   & 1         &                       &    & 1 \\
          &   &                                 &   &   &   &           &                       & \\
          &   & 1                               &   &   &   & 1         &                       & \\
          &   &                                 &   &   &   &           &                       & \\
          &   & 1                               &   &   &   & 1         &                       & \\
          &   & \mathclap{\fano_1(Q_1\cap Q_2)} &   &   & = & M_{4,1,0} & {\ }\quad + \quad{\ } &    & \mathclap{M_{4,1,1}[C]} &      & {\ }\quad + \quad{\ } &   &   & \mathclap{M_{4,1,2}[\Sym^2C]}
      \end{array}
    \end{equation}
    \caption{Decomposition of the Hodge diamond of $\fano_1(Q_1\cap Q_2)$ for~$g=4$}
    \label{figure:hodge-diamond-decomposition}
  \end{figure}
\end{example}

Finally, we can explain how \cref{theorem:hodge-decomposition}
can be used to give evidence for the effectivity aspect of~\cref{conjecture:hyperelliptic},
by using~$\mu_{\derived^\bounded}$,
for which~$\mu_{\derived^\bounded}(\lefschetz)=1$.
Applying Pascal's identity repeatedly
to the term~$\binom{2g-k-i}{k+1-i}$ of \eqref{equation:M}
after setting~$\mathbb{L}=1$
gives the following easy lemma.

\begin{lemma}
  \label{lemma:multiplicity-reduction}
  Setting~$\lefschetz=1$ in \eqref{equation:M} gives
  \begin{equation}
    M_{g,k,i}|_{\lefschetz=1}
    =
    \binom{2g-4-k-i}{k+1-i} + 2\binom{2g-4-k-i}{k-i}.
  \end{equation}
\end{lemma}
Thus after setting~$\mathbb{L}=1$ one gets a positive integer,
but we have not managed to prove rigorously that the expression~$M_{g,k,i}$ from \eqref{equation:M}
is effective already in~$\mathbb{Z}[\mathbb{L}]$.
It is a by-product of the experimental verification of \cref{conjecture:hyperelliptic-refined} in \cref{appendix:code} for many~$g$ and~$k$
that \eqref{equation:M} is indeed also effective in all the cases the conjecture has been numerically verified.

\section{Decomposing Fano schemes in the stacky case}
To give evidence for \cref{conjecture:stacky}
we will use the same method as in \cite[\S3]{MR4673249},
comparing the number of exceptional objects
on the right-hand side of \eqref{equation:stacky-conjecture}
to the Euler characteristic of~$\fano_k(Q_1\cap Q_2)$.
In \cref{subsection:cohomology-stacky,subsection:euler}
we will prove \cref{proposition:hodge-tate}
and \cref{proposition:euler-characteristic-even}
which concerns the left-hand side in \eqref{equation:stacky-conjecture}.
In \cref{subsection:stacky-symmetric-power}
we will recall the definition and properties of ``stacky symmetric powers'',
allowing us to compute the number of exceptional objects
on the right-hand side of \eqref{equation:stacky-conjecture}
in \cref{subsection:evidence-stacky}.

\subsection{Cohomology of the Fano scheme}
\label{subsection:cohomology-stacky}
As in the proof of \cref{theorem:hodge-decomposition},
we will determine the Hodge structure
on the cohomology of~$\fano_k(Q_1\cap Q_2)$,
and thus prove \cref{proposition:hodge-tate},
by tracking through the arguments in \cite{MR4216588}
and ensuring that they lift to the category of mixed Hodge modules.
As in op.~cit.~we will use~$n$ instead of~$g$.

\paragraph{Springer theory}
Let~$V$ be a~$\mathbb{C}$-vector space of dimension~$2n+1$.
We write~$G=\operatorname{SL}(V)$,
and~$K=\operatorname{SO}(V,q)$,
where~$q$ is a nondegenerate quadratic form on~$V$.
We write~$\mathfrak{g}$ for the Lie algebra of~$G$,
and~$\mathfrak{g}_1$ for the Lie algebra of~$K$,
respectively;
similarly, we write~$\mathcal{N}$ and~$\mathcal{N}_1$
for the nilpotent cones of~$G$ and~$K$, respectively.

We write $\mathfrak{g}^{\rs}$ for the regular semisimple elements of $\mathfrak{g}$,
and set $\mathfrak{g}^{\rs}_1 = \mathfrak{g}_1 \cap \mathfrak{g}^{\rs}$.
To each element~$\gamma \in \mathfrak{g}_1^{\rs}$,
one may associate the nondegenerate quadratic from $(\gamma -, -)$,
where $(-,-)$ is the bilinear form corresponding to $q$.

\paragraph{Upgrading to mixed Hodge modules}
We now describe a local system~$\mathcal{L}$ on~$\mathfrak{g}_1^{\rs}$.
For each $\gamma \in \mathfrak{g}_1^{\rs}$,
let $S_{\gamma} = \{a_0, \dots, a_{2n}\}$ be the branch points of the associated stacky curve.
We set $\mathcal{L}_{\gamma}$ to be
the free~$\mathbb{Z}$-module on the set $\{\pm a_0, \cdots, \pm a_{2g}\}/\pm 1$
of assignments $a_i \to \pm$, modulo a global sign.

As $\gamma$ varies
one obtains a local system $\mathcal{L}$ on $\mathfrak{g}_1^{\rs}$
of rank $2^{2n}$.
There exists a decomposition of local systems
\begin{equation}
  \mathcal{L} \cong \bigoplus_{i = 0}^g \mathcal{L}_i,
\end{equation}
where $\mathcal{L}_i$ is spanned by the elements $\{\pm a_i\}$ with $i$ plus signs if $i$ is even,
and $2n + 1 - i$ plus signs if $i$ is odd (up to the global sign).
Each $\mathcal{L}_i$ is a finite local system of rank $\binom{2n + 1}{i}$.
We regard $\mathcal{L}$ (and each $\mathcal{L}_i$) as a variation of Hodge structure
by declaring the stalks to be of Tate type and of weight $0$.

\begin{proposition}
  \label{proposition:key-even}
  There is an isomorphism of Hodge modules
  \begin{equation}
    \FL(\IC(\overline{\mathrm{O}}_{2^i 1^{2n - 2i + 1}}), \mathbb{Q}^\rH) \cong \IC(\mathfrak{g}_1, \mathcal{L}_i).
  \end{equation}
\end{proposition}

\begin{proof}
  We explain how to lift the proof of \cite[Proposition~3.1]{MR4216588} to Hodge modules.
  First, one considers families $v_n, \widecheck v_n$,
  which are analogous to the families
  in the hyperelliptic case considered previously.
  Arguing exactly as in the hyperelliptic case,
  one sees that there is an isomorphism of monodromic Hodge modules
  \begin{equation}
    \FL(v_{n,*} \mathbb{Q}^\rH[d_n]) \cong  \widecheck v_{n,*} \mathbb{Q}^\rH[\widecheck d_n](d_n - \widecheck d_n).
  \end{equation}
  Over the regular semisimple locus $\mathfrak{g}_1^{\rs}$,
  the family $\widecheck v_n$ is the relative Fano scheme of maximal linear subspaces
  in the intersection of two quadrics,
  as one of the quadrics varies over $\mathfrak{g}_1^{\rs}$.
  This is a finite covering space of degree $2^{2n}$,
  and the generic part of the right-hand side is the local system $\mathcal{L}$
  (up to shifts and twists).

  From the decomposition theorem, one sees that there is a decomposition of Hodge modules
  \begin{equation}
     \widecheck v_{n,*} \mathbb{Q}^\rH[\widecheck d_n] \cong \IC(\mathfrak{g}_1, \mathcal{L}) \cong \bigoplus_{i = 0}^n \IC(\mathfrak{g}_1, \mathcal{L}_i),
  \end{equation}
  where the terms on the right are simple Hodge modules, and we have suppressed twists.

  On the other hand, the decomposition theorem for $v_n$ gives
  \begin{equation}
     v_{n,*} \mathbb{Q}^\rH[d_i] \cong \bigoplus_{k = 0}^n \IC(\overline{\mathrm{O}}_{2^k 1^{2n - 2k + 1}}, \mathbb{Q}^\rH),
  \end{equation}
  and the terms on the right are simple.

  Therefore, there is an isomorphism of Hodge modules (with twists suppressed)
  \begin{equation}
    \bigoplus_{k = 0}^n \FL(\IC(\overline{\mathrm{O}}_{2^k 1^{2n - 2k + 1}}, \mathbb{Q}^\rH)) \cong \bigoplus_{i = 0}^n \IC(\mathfrak{g}_1, \mathcal{L}_i),
  \end{equation}
  and we need to show that the $k$th term on the left
  matches with the $i$th term on the right.
  Each local system $\mathcal{L}_i$ is distinguished by its rank,
  so it suffices to match terms on the level of perverse sheaves,
  which is done in \cite[\S4]{MR4216588}.
\end{proof}

\begin{proof}[Proof of \cref{proposition:hodge-tate}]
  With~\cref{proposition:key-even} in hand, one gets \cref{proposition:hodge-tate} by arguing exactly as in \cite[\S6.1]{MR4216588},
  and noting that the stalks of $\IC(\mathfrak{g}_1, \mathcal{L}_i)$ at elements of $\mathfrak{g}_1^{\rs}$ are of Tate type.
\end{proof}

\subsection{Euler characteristic of the Fano scheme}
\label{subsection:euler}
The goal of this subsection is to prove the following.
\begin{proposition}
  \label{proposition:euler-characteristic-even}
  Let~$Q_1\cap Q_2\subset\mathbb{P}^{2g}$ be a smooth intersection of smooth quadrics,
  where~$g\geq 2$.
  For~$k=0,\ldots,g-2$ we have
  \begin{equation}
    \label{equation:euler-characteristic-fano-scheme}
    \chi(\fano_k(Q_1\cap Q_2))
    =
    \binom{g}{k+1}4^{k+1},
  \end{equation}
  and~$\HHHH_i(\fano_k(Q_1\cap Q_2))=0$ for~$i\neq 0$,
  thus
  \begin{equation}
    \label{equation:zeroth-hochschild-homology}
    \dim\HHHH_0(\fano_k(Q_1\cap Q_2))
    =
    \binom{g}{k+1}4^{k+1}.
  \end{equation}
\end{proposition}
In \cite[\S3]{MR4673249} this is proven for~$k=g-2$ using a sequence of anti-flips
obtained by wall-crossing for the moduli space of quasiparabolic bundles.
We are not aware of a similar wall-crossing picture for~$\fano_k(Q_1\cap Q_2)$ when~$k\neq g-2$,
so we will rather build and expand upon the main result of \cite{MR4216588},
which leads to \cref{proposition:hodge-tate}.

\paragraph{A combinatorial identity}
We will use the following combinatorial lemma,
whose statement was suggested in \cite{ira-gessel}.
\begin{lemma}
  \label{lemma:ira-gessel}
  For all integers~$m\geq 0$ and~$a\geq 0$ even
  we have that
  \begin{equation}
    \label{equation:ira-gessel}
    4^m\binom{m+a/2}{m}
    =
    \sum_{i=0}^m\binom{m+a-i}{a}\binom{2m+a+1}{i}.
  \end{equation}
\end{lemma}

For completeness' sake we provide a proof of this statement,
using the method to prove a similar statement from \cite{1248154}.
\begin{proof}
  The proof strategy is via an application of formal power series.
  We apply the following identities of formal power series, the first of which is standard,
  and the second of which can be found in \cite[page~54]{MR2172781}:
  \begin{align}
    \label{equation:standard-power-series} \frac{1}{(1-x)^{n+1}} &= \sum_{i=0}^\infty \binom{n+i}{i}x^i \\
    \label{equation:gfology-power-series}\frac{1}{\sqrt{1-4x}}\left(\frac{1-\sqrt{1-4x}}{2x}\right)^j &= \sum_{i=0}^\infty \binom{2i+j}{i}x^i.
  \end{align}
  Note first that by switching the order of summation,
  the right side of \eqref{equation:ira-gessel} may be written as
  \begin{equation}
    \sum_{i=0}^m\binom{a+i}{a}\binom{2m+a+1}{m-i}.
  \end{equation}
  Now, to show the identity, it suffices to show the equality of the formal power series
  \begin{equation}
    F(x) = \sum_{m=0}^\infty 4^m\binom{m+a/2}{m}x^m, \qquad \text{ and }\qquad G(x) = \sum_{m=0}^\infty \sum_{i=0}^m\binom{a+i}{a}\binom{2m+a+1}{m-i}x^m,
  \end{equation}
  since this implies that all of their coefficients must agree.
  \Cref{equation:standard-power-series} immediately implies that
  \begin{equation}
    F(x) = \frac{1}{(1-4x)^{a/2+1}}.
  \end{equation}
  For the second power series, we switch the order of summation and reindex the sum by~$j=m-i$:
  \begin{align}
    G(x) &= \sum_{i=0}^\infty  \sum_{m=i}^\infty \binom{a+i}{a}\binom{2m+a+1}{m-i}x^m \\
    &=\sum_{i=0}^\infty \binom{a+i}{a}x^i \sum_{j=0}^\infty \binom{2j+2i+a+1}{j}x^j.
  \end{align}
  Then an application of \eqref{equation:gfology-power-series}, followed by an application of \eqref{equation:standard-power-series} gives
  \begin{align}
    G(x) &= \sum_{i=0}^\infty \binom{a+i}{a}x^i\cdot \frac{1}{\sqrt{1-4x}}\left(\frac{1-\sqrt{1-4x}}{2x}\right)^{2i+a+1} \\
    &= \frac{1}{\sqrt{1-4x}}\left(\frac{1-\sqrt{1-4x}}{2x}\right)^{a+1}\cdot\sum_{i=0}^\infty \binom{a+i}{a}x^i\cdot\left(\frac{1-\sqrt{1-4x}}{2x}\right)^{2i} \\
    &= \frac{1}{\sqrt{1-4x}}\cdot\left(\frac{1-\sqrt{1-4x}}{2x}\right)^{a+1}\cdot \frac{1}{\left(1-\left(\frac{1-\sqrt{1-4x}}{2x}\right)^2x\right)^{a+1}}.
  \end{align}
  It is then a straightforward exercise in algebra to verify that
  \begin{equation}
    G(x) = \frac{1}{(1-4x)^{a/2+1}} = F(x).
  \end{equation}
\end{proof}

\paragraph{Euler characteristic calculation}
For the proof of \cref{proposition:euler-characteristic-even} we need to recall some facts about Gaussian binomial coefficients.
We write
\begin{equation}
  \binom{m}{n}_q = \frac{\prod_{l=m-n+1}^m (1-q^l)}{\prod_{l=1}^n (1-q^l)}\in\mathbb{Z}[q]
\end{equation}
for the Gaussian binomial coefficient in the variable~$q$.
We will write~$[q^r] \binom{m}{n}_q$ for the coefficient of~$q^r$ in~$\binom{m}{n}_q$.
The key facts we need are that
\begin{equation}
  \label{equation:q-binomial-symmetry}
  \binom{m}{n}_q =  \binom{m}{m-n}_q
\end{equation}
and
\begin{equation}
  \label{equation:q-binomial-limit}
  \lim_{q\to 1} \binom{m}{n}_q = \binom{m}{n}.
\end{equation}

We are now ready to prove the identity of Euler characteristics claimed above.

\begin{proof}[Proof of \cref{proposition:euler-characteristic-even}]
  By \cref{proposition:hodge-tate},
  and reinterpreting the multiplicities in that statement using Gaussian binomial coefficients,
  we have that
  \begin{equation}
    \mathrm{b}_{2p}(\fano_k(Q_1\cap Q_2))
    =
    \sum_{j=0}^{k+1} \binom{2g+1}{j}\cdot[q^{p-j(g-k-1)}]\binom{2g-k-j-1}{k+1-j}_q.
  \end{equation}
  Thus we have that
  \begin{equation}
    \begin{aligned}
      \chi(\fano_k(Q_1\cap Q_2))
      &=\sum_{p=0}^{\dim\fano_k(Q_1\cap Q_2)}\mathrm{b}_{2p}(\fano_k(Q_1\cap Q_2)) \\
      &=\sum_{p=0}^{\dim\fano_k(Q_1\cap Q_2)}\sum_{j=0}^{k+1}\binom{2g+1}{j}\cdot[q^{p-j(n-k-1)}]\binom{2g-k-j-1}{k+1-j}_q \\
      &=\sum_{j=0}^{k+1}\binom{2g+1}{j}\sum_{p=0}^{\dim\fano_k(Q_1\cap Q_2)}[q^{p-j(n-k-1)}]\binom{2g-k-j-1}{2(g-k-1)}_q.
    \end{aligned}
  \end{equation}
  By \cref{proposition:fano-scheme-even-dimension} we have $\dim\fano_k(Q_1\cap Q_2) = 2(k+1)(g-k-1)$.
  We also have
  \begin{equation}
    \deg_q\binom{2g-k-j-1}{2(g-k-1)}_q=2(k+1-j)(g-k-1).
  \end{equation}
  Thus, the inner sum runs through all the non-zero coefficients of $\smash{\binom{2g-k-j-1}{2(g-k-1)}_q}$.
  Now, using \eqref{equation:q-binomial-symmetry}, \eqref{equation:q-binomial-limit} and~\cref{lemma:ira-gessel}
  with $m=k+1$ and $a=2(g-k-1)$ we obtain that
  \begin{equation}
    \chi(\fano_k(Q_1\cap Q_2))
    =\sum_{j=0}^{k+1} \binom{2g+1}{j} \binom{2g-k-j-1}{2(g-k-1)}
    =4^{k+1} \binom{g}{k+1}.
  \end{equation}
  Because the Hodge diamond of~$\fano_k(Q_1\cap Q_2)$ is concentrated on the vertical axis,
  we have as in \cref{remark:comment-on-euler-characteristic}
  that
  \begin{equation}
    \chi(\fano_k(Q_1\cap Q_2))
    =
    \dim\HHHH_\bullet(\fano_k(Q_1\cap Q_2))
    =
    \dim\HHHH_0(\fano_k(Q_1\cap Q_2)).
  \end{equation}
\end{proof}

\begin{remark}
  For~$k=g-1$ we have that~$\fano_{g-1}(Q_1\cap Q_2)$
  is a reduced and finite scheme of cardinality~$2^{2g}$ \cite[Theorem~3.8]{reid-thesis},
  thus the formula \eqref{equation:euler-characteristic-fano-scheme} still holds in this degenerate case.
\end{remark}

\subsection{The ``stacky symmetric power''}
\label{subsection:stacky-symmetric-power}

We introduce the following notation.
Associated to~$p\in\mathbb{P}^1$
we have by projective duality an associated~$p^\vee\in\mathbb{P}^{1,\vee}\cong\mathbb{P}^1$.
The~$k$th Veronese embedding for~$\mathbb{P}^{1,\vee}$ gives a closed immersion~$\mathbb{P}^{1,\vee}\hookrightarrow\mathbb{P}^{k,\vee}$.
The image of~$p^\vee$ gives a point~$h^\vee(p)\in\mathbb{P}^{k,\vee}$
which corresponds to a hyperplane~$H(p)\subset\mathbb{P}^k$.

The image of the Veronese embedding is the rational normal curve,
hence the points~$p_i^\vee$ are in general position,
meaning that the hyperplanes~$H(p_i)$ form a generalized snc divisor:
for all~$I\subset\{1,\ldots,2g+1\}$ and~$j\in I$ the inclusion
\begin{equation}
  \bigcap_{i\in I}H(p_i)\hookrightarrow\bigcap_{i\in I\setminus\{j\}}H(p_i)
\end{equation}
is an effective Cartier divisor.

With this notation,
Fonarev gives the following ad hoc definition of a stacky symmetric power \cite[\S2]{MR4673249}.
\begin{construction}
  \label{construction:fonarev}
  Let~$\mathcal{C}$ be the stacky curve~$\sqrt[2]{\mathbb{P}^1,p_1+\dots+p_{2g+1}}$
  attached to~$Q_1\cap Q_2\subset\mathbb{P}^{2g}$.
  We define its \emph{``stacky symmetric power''} as the iterated root stack
  \begin{equation}
    \fonarev{k}
    \colonequals
    \sqrt[2]{\mathbb{P}^k,H(p_1)+\dots+H(p_{2g+1})},
  \end{equation}
  i.e.,
  the iterated fiber product of the root stacks~$\sqrt[2]{\mathbb{P}^k,H(p_i)}$ over~$\mathbb{P}^k$.
\end{construction}

\begin{remark}
  It would be interesting to find a general moduli-theoretic interpretation of~$\fonarev{k}$,
  which reduces to the symmetric power of a curve if there were no stacky points,
  and which works for an arbitrary stacky curve.
  We are not aware of such an interpretation.
\end{remark}

Let us recall the following lemma, which is proven in \cite[\S3]{MR4673249}
using the semiorthogonal decomposition of a root stack obtained in \cite[Theorem~1.6]{MR3436544}
and \cite[Theorem~4.7]{MR3573964}.
\begin{lemma}
  \label{lemma:fonarev}
  Let~$Q_1\cap Q_2\subset\mathbb{P}^{2g}$ be a smooth intersection of smooth quadrics.
  Let~$\mathcal{C}$ be the associated stacky curve.
  Then~$\derived^\bounded(\fonarev{k})$
  admits a full exceptional collection of length
  \begin{equation}
    \rk\Kzero(\derived^\bounded(\fonarev{k})) =\sum_{t=0}^k(k+1-t)\binom{2g+1}{t},
  \end{equation}
  where~$\fonarev{k}$ is as in \cref{construction:fonarev}.
\end{lemma}

\subsection{Evidence for \texorpdfstring{\cref{conjecture:stacky}}{Conjecture \ref{conjecture:stacky}}}
\label{subsection:evidence-stacky}
To give evidence for \cref{conjecture:stacky}
we will compute the number of exceptional objects
in the right-hand side of \eqref{equation:stacky-conjecture},
and compare it to the zeroth Hochschild homology of~$\fano_k(Q_1\cap Q_2)$
as computed in \cref{proposition:euler-characteristic-even}.
The following generalises \cite[\S3]{MR4673249}
from~$k=g-2$ to all~$k=0,\ldots,g-2$.
\begin{proposition}
  \label{proposition:number-of-objects-rhs}
  The right-hand side of \eqref{equation:stacky-conjecture} in \cref{conjecture:stacky}
  admits a full exceptional collection
  of length~$\binom{g}{k+1}4^{k+1}$.
\end{proposition}

\begin{proof}[Proof of \cref{proposition:number-of-objects-rhs}]
  By \cref{lemma:fonarev}
  the right-hand side of \eqref{equation:stacky-conjecture}
  is a semiorthogonal decomposition where each component has a full exceptional collection,
  thus it admits itself a full exceptional collection.
  It remains to show that
  for all~$g\geq 2$ and~$k\leq g-2$ its total length is
  \begin{equation}
    \label{equation:stacky-rhs-goal}
    \binom{g}{k+1}4^{k+1}
    =
    \sum_{i=0}^{k+1}
    \sum_{t=0}^i
    \binom{2g-3-k-i}{k+1-i}(i+1-t)\binom{2g+1}{t}.
  \end{equation}
  Taking~$m=k+1$ and~$a=2g-2k-2$ in \cref{lemma:ira-gessel} we obtain
  \begin{equation}
    4^{k+1}\binom{g}{k+1}
    =
    \sum_{t=0}^{k+1}\binom{2g-k-1-t}{2g-2k-2}\binom{2g+1}{t}.
  \end{equation}
  Using the variant of the Chu--Vandermonde identity which reads
  \begin{equation}
    \sum_{m=0}^n \binom{m}{j}\binom{n-m}{k-j}
    =
    \binom{n+1}{k+1}
  \end{equation}
  for~$0\leq j\leq k\leq n$ we can write
  \begin{equation}
    \binom{2g-k-1-t}{2g-2k-2}
    =
    \sum_{\alpha=0}^{2g-k-t}\binom{\alpha}{1}\binom{2g-k-2-t-\alpha}{2g-2k-4}.
  \end{equation}
  The term~$\binom{\alpha}{1}$ vanishes for~$\alpha=0$,
  so we may index the sum starting from~$\alpha=1$.
  Now if we reindex the summation by taking~$i=\alpha+t-1$, we find
  \begin{equation}
    \begin{aligned}
      \binom{2g-k-1-t}{2g-2k-2}
      &=\sum_{i=t}^{2g-k-1}\binom{i+1-t}{1}\binom{2g-3-k-i}{2g-2k-4} \\
      &=\sum_{i=t}^{2g-k-1}\binom{i+1-t}{1}\binom{2g-3-k-i}{k+1-i} \\
      &=\sum_{i=t}^{k+1}\binom{i+1-t}{1}\binom{2g-3-k-i}{k+1-i}
    \end{aligned}
  \end{equation}
  where the last equality follows since the second term in the product vanishes for~$i > k + 1$.
  We therefore have the equality
  \begin{equation}
    \begin{aligned}
      \binom{g}{k+1}4^{k+1}
      &=\sum_{t=0}^{k+1}\sum_{i=t}^{k+1}\binom{2g-3-k-i}{k+1-i}(i+1-t) \\
      &=\sum_{i=0}^{k+1}\sum_{t=0}^{i}\binom{2g-3-k-i}{k+1-i}(i+1-t),
    \end{aligned}
  \end{equation}
  as in \eqref{equation:stacky-rhs-goal}.
\end{proof}

It thus agrees with \cref{proposition:euler-characteristic-even}.

\appendix

\section{Code to verify \texorpdfstring{\cref{conjecture:hyperelliptic-refined}}{Conjecture \ref{conjecture:hyperelliptic-refined}} after applying the E-polynomial motivic measure}
The following SageMath code uses \cite{hodge-diamond-cutter},
commit \texttt{1950acc},
and is tested with SageMath~10.2.
It is available at \url{https://github.com/pbelmans/decomposing-fano-schemes},
which also gives a GitHub Action that automatically runs the code.
\label{appendix:code}
\inputminted[fontsize=\small]{sage}{code/paper.sage}

\printbibliography

\emph{Pieter Belmans}, \url{pieter.belmans@uni.lu} \\
Department of Mathematics, Universit\'e de Luxembourg, 6, avenue de la Fonte, L-4364 Esch-sur-Alzette, Luxembourg

\emph{Jishnu Bose}, \url{jishnubo@usc.edu} \\
Department of Mathematics, University of Southern California, 3620 Vermont Ave, Los Angeles, CA 90089, United States

\emph{Sarah Frei}, \url{sarah.frei@dartmouth.edu} \\
Department of Mathematics, Dartmouth College, Hanover, NH 03755, United States

\emph{Ben Gould}, \url{brgould@umich.edu} \\
Department of Mathematics, University of Michigan -- Ann Arbor, 530 Church St, Ann Arbor, MI 48104, United States

\emph{James Hotchkiss} \url{james.hotchkiss@columbia.edu} \\
Department of Mathematics, Columbia University, 2990 Broadway, New York, NY 10027, United States

\emph{Alicia Lamarche}, \url{lamarche@math.utah.edu} \\
Department of Mathematics, University of Utah, 155 South 1400 East JWB 233, Salt Lake City, UT 84112, United States

\emph{Jack Petok}, \url{jack.petok@dartmouth.edu} \\
Department of Mathematics, Dartmouth College, Hanover, NH 03755, United States

\emph{Cristian Rodriguez Avila}, \url{rodriguez@math.umass.edu} \\
Department of Mathematics and Statistics, University of Massachusetts, 710 N. Pleasant Street, Amherst, MA 01003, United States

\emph{Saket Shah}, \url{sakets@umich.edu} \\
Department of Mathematics, University of Michigan -- Ann Arbor, 530 Church St, Ann Arbor, MI 48104, United States

\end{document}